\documentclass[final, reqno, 11pt]{amsart}
\usepackage{amsmath}
\usepackage{amsfonts}
\usepackage{amssymb}
\usepackage{amsthm}
\thanks{CAMS, EHESS, 54 bd Raspail, Paris 75006, France ({\tt mbeceanu@ehess.fr}). The author is partly supported by the ANR grant PREFERED. This work stems from the author's Ph.D.\ thesis at the University of Chicago.}

\newtheorem{theorem}{Theorem}[section]
\newtheorem{lemma}[theorem]{Lemma}
\newtheorem{proposition}[theorem]{Proposition}
\newtheorem{corollary}[theorem]{Corollary}
\theoremstyle{remark}
\newtheorem{observation}[theorem]{Remark}

\newtheorem{definition}{Definition}[section]
\newcommand{\set}{\mathbb}
\newcommand{\dl}{\nabla}

\newcommand{\mc}{\mathcal}
\newcommand{\be}{\begin{equation}}
\newcommand{\ee}{\end{equation}}
\newcommand{\bee}{\begin{align}}
\newcommand{\eee}{\end{align}}

\newcommand{\ba}{\begin{array}}

\newcommand{\ea}{\end{array}}
\newcommand{\bpm}{\begin{pmatrix}}
\newcommand{\epm}{\end{pmatrix}}
\newcommand{\lb}{\label}
\DeclareMathOperator{\sgn}{sgn}

\DeclareMathOperator{\supp}{supp}
\DeclareMathOperator{\Ran}{Ran}
\DeclareMathOperator{\Ker}{Ker}
\DeclareMathOperator{\Rere}{Re}
\DeclareMathOperator{\Imim}{Im}
\DeclareMathOperator{\Dom}{Dom}
\DeclareMathOperator*{\esssup}{ess\; sup}

\newcommand{\ov}{\overline}
\newcommand{\dd}{{\,}{d}}

\renewcommand{\Re}{\Rere}
\renewcommand{\Im}{\Imim}

\title
{New Estimates for a Time-Dependent Schr\"{o}dinger Equation}
\author{Marius Beceanu}
\date{September 2009}
\begin{document}
\maketitle
\numberwithin{equation}{section}
\begin{abstract}
This paper establishes new estimates for linear Schr\"{o}\-din\-ger equations in $\set R^3$ with time-dependent potentials. Some of the results are new even in the time-independent case and all are shown to hold for potentials in scaling-critical, translation-invariant spaces.

The proof of the time-independent results uses a novel method based on an abstract version of Wiener's Theorem.

\end{abstract}


\section{Introduction}
\subsection{Overview}
Consider the linear Schr\"{o}dinger equation in $\set R^3$
\be\lb{1.1}
i \partial_t Z + \mc H Z = F,\ Z(0) \text{ given},
\ee
where
\be\lb{eq_3.2}
\mc H = \mc H_0 + V = -\Delta + V,
\ee
where $V$ may be real or complex-valued, in the scalar case and
\be
\mc H = \mc H_0 + V = \bpm \Delta-\mu & 0 \\ 0 & -\Delta+\mu \epm + \bpm W_1 & W_2 \\ -\ov W_2 & -W_1 \epm,\ \mu>0
\lb{3.2}
\ee
in the matrix nonselfadjoint case. $W_1$ is always taken to be real-valued, while $W_2$ may also be complex-valued. We treat all cases in a unified manner.

Let $R_0(\lambda) = (\mc H_0 - \lambda)^{-1}$ be the resolvent of the unperturbed Hamiltonian,
\be\lb{eq_3.55}
R_0(\lambda^2)(x, y) = \frac 1 {4\pi} \frac {e^{i \lambda |x-y|}}{|x-y|}
\ee
in the scalar case (\ref{eq_3.2}) and
\be\lb{eq_3.56}
R_0(\lambda^2+\mu)(x, y) = \frac 1 {4\pi} \bpm -\frac {e^{-\sqrt{\lambda^2+2\mu} |x-y|}}{|x-y|} & 0 \\ 0 & \frac {e^{i \lambda |x-y|}}{|x-y|} \epm
\ee
in the matrix case (\ref{3.2}). For a multiplicative decomposition $V=V_1 V_2$, such that $V_2 R_0(\lambda) V_1 \in \mc L(L^2, L^2)$ (bounded from $L^2$ to itself), the exceptional set $\mc E$ is defined as the set of $\lambda \in \set C$ such that $(I + V_2 R_0(\lambda) V_1)^{-1} \not \in \mc L(L^2, L^2)$.

Throughout this paper we make the simplifying assumption that $\mc E$ is disjoint from the spectrum of $\mc H_0$. In $\set R^3$, this assumption holds generically. The opposite situation requires separate treatment.


In proving our estimates for the solution $Z$ of (\ref{1.1}), we use an abstract version of Wiener's Theorem, which also presents independent interest.
\begin{theorem}\lb{thm_1.6}
Let $H$ be a Hilbert space and $K = \mc L(H, M_t H)$ be the algebra of bounded operators from $H$ to $M_t H$, where $M_t H$ is the space of $H$-valued Borel measures on $\set R$ of finite mass, see (\ref{3.11}).

If $A \in K$ is invertible then $\widehat A(\lambda)$ is invertible for every $\lambda$. Conversely, assume $\widehat A(\lambda)$ is invertible for each $\lambda$, $A = I + L$, and
\be
\lim_{\epsilon \to 0} \|L(\cdot + \epsilon) - L\|_K = 0,\ \lim_{R \to \infty}\|(1-\chi(t/R)) L(t)\|_K =0
\ee
Then $A$ is invertible in $K$.
\end{theorem}
\noindent Moreover, if $A$ is supported on $[0, \infty)$, then so is $A^{-1}$, by Lemma~\ref{lemma8}.

Another helpful result is the following: 
\begin{proposition}\lb{prop_21} Let $\mc H_0$ be as in (\ref{eq_3.2}) or (\ref{3.2}).
Then
\be\lb{eqn_3.39}
\int_{\set R} \|e^{it \mc H_0} f\|_{L^{6, \infty}} \dd t \leq C \|f\|_{L^{6/5, 1}}.
\ee
\end{proposition}
We actually prove the stronger inequality
\be\lb{eq_3.39}
\sum_{n \in \set Z} 2^n \sup_{t \in [2^n, 2^{n+1})} \|e^{it \mc H_0} f\|_{L^{6, \infty}} \leq C \|f\|_{L^{6/5, 1}}.
\ee
We use Proposition \ref{prop_21} in the context of Theorem \ref{thm_1.6} --- to show, for specific weights $V_1$ and $V_2$, that $V_2 e^{it \mc H_0} V_1$ is an operator-valued measure of finite mass.

Finally, let $L^{3/2, \infty}_0$ be the weak-$L^{3/2}$ closure of the set of bounded compactly supported functions. The following is the main result for time-independent potentials:
\begin{theorem}\lb{theorem_26} Let $Z$ be a solution of the linear Schr\"{o}\-din\-ger equation (\ref{1.1})
$$
i \partial_t Z + \mc H Z = F,\ Z(0) \text{ given}.
$$
Assume that $\mc H = \mc H_0 + V$, $V$ is as in (\ref{eq_3.2}) or (\ref{3.2}), and that no exceptional values of $\mc H$ are contained in $\sigma(\mc H_0)$. Then Strichartz estimates hold: if $V \in L^{3/2, \infty}_0$ then
\be\lb{1.45}
\|P_c Z\|_{L^{\infty}_t L^2_x \cap L^2_t L^{6, 2}_x} \leq C \Big(\|Z(0)\|_2 + \|F\|_{L^1_t L^2_x + L^2_t L^{6/5, 2}_x}\Big)
\ee
and if $V \in L^{3/2, 1}$ then
\be\lb{1.46}
\|P_c Z\|_{L^1_t L^{6, \infty}_x} \leq C \Big(\|Z(0)\|_{L^{6/5, 1}} + \|F\|_{L^1_t L^{6/5, 1}_x}\Big).
\ee
\end{theorem}
Here $P_c$ is the projection on the continuous spectrum of $\mc H$.
%

Equation (\ref{1.1}) is invariant under rescaling: replacing $Z(t, x)$ in this equation by $Z(\alpha^2 t, \alpha x)$, $V(x)$ by $\alpha^2 V(\alpha x)$, and $F(t, x)$ by $F(\alpha^2 t, \alpha x)$, the equation remains valid. It is natural to study (\ref{1.1}) under scaling-invariant norms.


Our results are scaling-invariant in that $V$ appears with scaling-invariant norms and in that the constants in (\ref{eqn_3.39}), (\ref{1.45}), and (\ref{1.46}) remain the same after rescaling. 

Using Lorentz spaces instead of the more usual Sobolev spaces is essential in the statements and proofs of (\ref{eqn_3.39}) and (\ref{1.46}). It also allows an improvement in (\ref{1.45}) --- from $L^{3/2}$, proved in \cite{gol}, to weak-$L^{3/2}$, the weakest space of potentials for which Strichartz estimates are known so far. Estimate (\ref{1.46}) is entirely new.


Note that if $V$ is scalar, real-valued, and in $L^{3/2}$, then it is enough to assume that zero is not an eigenvalue or a resonance of $\mc H$. This follows from \cite{ionjer}. Proposition \ref{prop28} summarizes the known spectral properties of~$\mc H$ if $V \in L^{3/2, \infty}_0$.

Overall, Lorentz spaces provide a convenient unified framework for our methods, some of which are based on real interpolation (see the Appendices).


\begin{observation}
For $Z(0)=0$, the dual of (\ref{1.46}) is
\be
\|P_c Z\|_{L^{\infty}_t L^{6, \infty}_x} \leq C \|F\|_{L^{\infty}_t L^{6/5, 1}_x}.
\ee
Real interpolation between (\ref{1.45}) and (\ref{1.46}) produces
\be
\|P_c Z\|_{L^p_t L^{6, p}_x} \leq C \big(\|Z(0)\|_{L^{q, p}} + \|F\|_{L^1_t L^{q, p}_x + L^p_t L^{6/5, p}_x}\big),
\ee
where $1 \leq p \leq 2$ and $3/(2q) = 1/p + 1/4$.
\end{observation}

In particular, this estimate and its dual have a simple form for $p=6/5$:
\begin{proposition}\lb{cor1.5} For $Z(0)=0$ and $V \in L^{3/2}$, provided zero is not a resonance or eigenvalue for $\mc H = \mc H_0 + V$,
\be\lb{eqn1.12}
\|P_c Z\|_{L^{6/5}_t L^{6, 6/5}_x} \leq C \|F\|_{L^{6/5}_{t, x}};\ \|P_c Z\|_{L^6_{t, x}} \leq C \|F\|_{L^6_t L^{6/5, 6}_{t, x}}.
\ee
\end{proposition}

Other estimates, such as $t^{-3/2}$ decay estimates or wave operator estimates, will be the subject of separate papers, such as \cite{bego} or \cite{bec3}. These results make use of Wiener's Theorem, Theorem \ref{thm_1.6}, in an even more general form.

\subsection{Time-dependent potentials}
We turn to time-dependent potentials. The difference is that, in this case, the Fourier transform of a kernel $T(t, s)$ which is not invariant under time translation is no longer a multiplier $\widehat T(\lambda):H \to H$ for each $\lambda$; it is a family of non-local operators instead. Such a generalization was studied by Howland \cite{how}.

We shall not follow this direction here; instead, we look only at perturbations of operators that are invariant under time translation, whose Fourier transform is then a perturbation of a multiplier operator.

Our results in the time-dependent case are, to some extent, independent from those for time-independent potentials. If Strichartz estimates or (\ref{1.46}) hold, no matter how they were obtained, then we can extend them to the case of time-dependent potentials.

Theorem \ref{theorem_26} implies the following easy generalization of itself:
\begin{corollary}\lb{cor_1.5}
Let $Z$ be a solution of
\be\lb{1.12}
i \partial_t Z + \mc H Z + \tilde V(t, x) P_c Z= F,\ Z(0) \text{ given}.
\ee
Let $\mc H = \mc H_0 + V$ be such that no exceptional values of $\mc H$ are contained in $\sigma(\mc H_0)$. If $V \in L^{3/2, \infty}_0$ and if $\|\tilde V\|_{L^{\infty}_t L^{3/2, \infty}_x}$ is sufficiently small, then
\be\lb{1.13}
\|P_c Z\|_{L^{\infty}_t L^2_x \cap L^2_t L^{6, 2}_x} \leq C \Big(\|Z(0)\|_2 + \|F\|_{L^1_t L^2_x + L^2_t L^{6/5, 2}_x} \Big).
\ee
If $V \in L^{3/2, 1}$ and if $\|\tilde V\|_{L^{\infty}_t L^{3/2, 1}_x}$ is sufficiently small, then
\be\lb{1.14}
\|P_c Z\|_{L^1_t L^{6, \infty}_x} \leq C \Big(\|Z(0)\|_{L^{6/5, 1}} + \|F\|_{L^1_t L^{6/5, 1}_x} \Big).
\ee
\end{corollary}
In other words, we can add a small time-dependent potential term, which depends in an arbitrary manner on both $x$ and $t$.

If instead of $\tilde V(t, x) P_c Z$ we have $\tilde V(t, x) Z$, this term may interact with the point spectrum, so we need to control $P_p Z$ in the same norm: for example,
\be
\|P_c Z\|_{L^1_t L^{6, \infty}_x} \leq C \Big(\|Z(0)\|_{L^{6/5, 1}} + \|F\|_{L^1_t L^{6/5, 1}_x} + \|P_p Z\|_{L^1_t L^{6, \infty}_x} \Big).
\ee

The situation is different for large time-dependent potentials, which may destroy dispersive estimates. However, we shall study a particular kind of time-dependent potentials, for which dispersive estimates still hold.

Consider the following manifold of unitary transformations $U: \set R^4 \times SO(3) \to \mc L(L^2, L^2)$, given by a combination of translations, rotations, and change of complex phase:
\be\lb{1.20}
U(D, A, \Omega) = \Omega e^{D \dl + i A \sigma_3}
\ee
in the matrix case and 
\be\lb{1.21}
U(D, A, \Omega) = \Omega e^{D \dl + i A}
\ee
in the scalar case. Here $D$ represents the translation, $\Omega$ the rotation, and $A$ the change of complex phase.

Within this manifold, we consider a one-parameter family of transformations $U(t):=U(\pi(t))$, governed by a parameter path $\pi(t)$:
\be\lb{1.26}
\pi:\set R \to \set R^4 \times SO(3),\ \pi(t) = (D(t), A(t), \Omega(t)),
\ee
where $D(t) \in \set R^3$, $A(t) \in \set R$, and $\Omega(t) \in SO(3)$. Thus,
\be
U(t) = \Omega(t) e^{D(t) \dl + i A(t) \sigma_3},\ \text{respectively } U(t) = \Omega(t) e^{D(t) \dl + i A(t)}.
\ee
In the sequel, we assume that the derivative of $\pi$ is essentially bounded: $\|\pi'\|_{\infty} < \infty$. In particular, $SO(3)$ is a Lie group, endowed with a canonical Riemannian metric; we use it to measure $\|\Omega'\|_{\infty}$.

We plug this family of unitary transformations into the equation (\ref{1.1}) as follows, obtaining a time-dependent potential:
\be\lb{3.157}
i \partial_t R(t) + (\mc H_0 + U(t)^{-1} V U(t)) R(t) = F(t),\ R(0) \text{ given}.
\ee
The Hamiltonian at time $t$ is a conjugate of $\mc H$:
\be
\mc H_0 + U(t)^{-1} V U(t) = U(t)^{-1} \mc H U(t) = U(t)^{-1} (\mc H_0 + V) U(t),
\ee
since $\mc H_0$ and $U$ commute.

The change of phase $e^{iA}$ is trivial in the scalar case, since it commutes with $V$, but $e^{iA\sigma_3}$ is nontrivial in the matrix case, since $e^{iA\sigma_3} V \ne V e^{iA\sigma_3}$.

If $U$ is of the form (\ref{1.20}), then (\ref{3.157}) becomes
\be\lb{129}
i \partial_t R(t) + (\mc H_0 + V(t, x)) R(t) = F(t),\ R(0) \text{ given},
\ee
where
\be\lb{130}
V(t, x) = \bpm W_1(\Omega(t) x- \Omega(t) D(t)) & e^{i A(t)} W_2(\Omega(t) x- \Omega(t) D(t)) \\ -e^{-iA(t)} \ov W_2(\Omega(t) x- \Omega(t) D(t)) & -W_1(\Omega(t) x - \Omega(t)D(t)) \epm.
\ee
The potential $V(x, t)$ has the same shape at all times $t$, but can rotate or change position and complex phase.

In the scalar case, the equation looks the same, but $V(t, x)$ is given by
\be\lb{131}
V(t, x) = V(\Omega(t) x - \Omega(t)D(t)),
\ee
so $A(t)$ does not even enter the equation.

Assume that $\pi$ is continuous and a.e.\ differentiable. Then, denote $Z(t) = U(t) R(t)$ and rewrite the equation in the new variable $Z$:
\be\lb{3.158}
i \partial_t Z(t) - i \partial_t U(t) U(t)^{-1} Z(t) + \mc H_0 Z(t) + V Z(t) = U(t) F(t),\ Z(0) = U(0) R(0).
\ee

In case $U$ is given by (\ref{1.21}) and $\Omega(t)=I$, (\ref{3.158}) becomes
\be\lb{1.30}
i \partial_t Z - i D'(t) \dl Z + A'(t) \sigma_3 Z + \mc H Z = U F,\ Z(0) \text{ given}.
\ee
In this formulation the potential is constant, plus the higher-order perturbation  $- i D'(t) \dl + A'(t) \sigma_3$.

More generally, we can also consider a manifold of transformations $U(\pi)$, a one-parameter path $U(t):=U(\pi(t))$ within it, and construct a time-dependent potential like above. The subsequent theorem is formulated abstractly, for such a general situation, but we always keep in mind the concrete cases (\ref{1.20}) and (\ref{1.21}).


The main abstract result in the time-de\-pen\-dent setting is the following:
\begin{theorem}\lb{theorem_13}
Consider equation (\ref{3.157}), for $\mc H = \mc H_0 + V$ as in (\ref{eq_3.2}) or (\ref{3.2}) and $V$ in $L^{3/2, \infty}_0$, not necessarily real-valued:
$$
i \partial_t R(t) + (\mc H_0 + U(t)^{-1} V U(t)) R(t) = F(t),\ R(0) \text{ given}.
$$
Let $P_c$ be the continuous spectrum projection of $\mc H$ and
\be
P_c(t) = U(t)^{-1} P_c U(t),\ P_p(t) = U(t)^{-1} P_p U(t) = I - P_c(t).
\ee
Assume that $U(t)$, $t \in \set R$, is a family of maps determined by an a.e.\  differentiable parameter path $\pi(t)$, with the following properties:
\begin{enumerate}
\item[P1] $U(t)$ and $U(t)^{-1}$ are uniformly $L^p$-bounded maps, for $1 \leq p \leq\infty$.
\item[P2] For every $t \in \set R$, $U(t)$ commutes with $\mc H_0$.
\item[P3] For some $N$ and for almost every $\tau \in \set R$
\be\lb{3.163}\begin{aligned}
\lim_{\substack{\|\pi'\|_{\infty} \to 0}} \sup_{t-s=\tau} \Big\|\langle x \rangle^{-N} \big(U(t) U(s)^{-1} e^{i(t-s) \mc H_0} - e^{i(t-s) \mc H_0}\big) \langle x \rangle^{-N}\Big\|_{2 \to 2} = 0.
\end{aligned}\ee
\item[P4] For any eigenfunction $f$ of $\mc H$ or $\mc H^*$, $\|\partial_t U(t) U^{-1}(t) f\|_{L^{6/5, 2}} \leq C \|\pi'\|_{\infty}$.
\end{enumerate}

Assume that $\|\pi'\|_{\infty}$ is sufficiently small (in a manner that depends on $V$) and there are no exceptional values of $\mc H$ embedded in $\sigma(\mc H_0)$. Then, one has
\be
\|P_c(t) R(t)\|_{L^{\infty}_t L^2_x \cap L^2_t L^{6, 2}_x} \leq C \Big(\|R(0)\|_2 + \|F\|_{L^1_t L^2_x + L^2_t L^{6/5, 2}_x} + \|P_p(t) R(t)\|_{L^2_t L^{6, 2}_x}\Big).
\ee
Further assume that $V \in L^{3/2, 1}$ and P4 holds with $L^{6/5, 1}$.
Then
\be\lb{2.157}
\|P_c(t) R(t)\|_{L^1_t L^{6, \infty}_x} \leq C \Big(\|R(0)\|_{L^{6/5, 1}} + \|F\|_{L^1_t L^{6/5, 1}_x} + \|P_p(t) R(t)\|_{L^1_t L^{6, \infty}_x}\Big).
\ee
\end{theorem}

By Lemma \ref{lem_32}, we show that $U(t)$ given by (\ref{1.20}) or (\ref{1.21}) and (\ref{1.26}) has properties P1--P4, so Theorem \ref{theorem_13} applies. In particular, we obtain
\begin{corollary}\lb{cor_1.7}
Consider a solution of (\ref{129}), where $V(x, t)$ is in the form (\ref{130}) or (\ref{131}).
Assume that $V \in L^{3/2, \infty}_0$, $\mc H$ has no exceptional values in $(-\infty, -\mu] \cup [\mu, \infty)$, and that $\|\Omega'\|_{L^{\infty}_t}$, $\|D'\|_{L^{\infty}_t}$, and $\|A'\|_{L^{\infty}_t}$ are sufficiently small, in a manner that depends on $V$. Then
\be\begin{aligned}
\|P_c(t) R\|_{L^{\infty}_t L^2_x \cap L^2_t L^{6, 2}_x} \leq C \Big(\|R(0)\|_{L^2} + \|F\|_{L^1_t L^2_x + L^2_t L^{6/5, 2}_x} + \|P_p(t) R\|_{L^2_t L^{6, 2}_x}\Big).
\end{aligned}\ee
If $V \in L^{3/2, 1}$, then
\be
\|P_c(t) R\|_{L^1_t L^{6, \infty}_x} \leq C \Big(\|R(0)\|_{L^{6/5, 1}} + \|F\|_{L^1_t L^{6/5, 1}_x} + \|P_p(t) R\|_{L^1_t L^{6, \infty}_x}\Big).
\ee
\end{corollary}
We also formulate the analogous conclusion for (\ref{1.30}):
\begin{corollary}\lb{cor_1.8}
Consider a solution of (\ref{1.30}).
Assume that $V \in L^{3/2, \infty}_0$, that $\mc H$ has no exceptional values in $(-\infty, -\mu] \cup [\mu, \infty)$, and that $\|D'\|_{L^{\infty}_t}$ and $\|A'\|_{L^{\infty}_t}$ are sufficiently small, in a manner that depends on $V$. Then
\be\begin{aligned}
\|P_c Z\|_{L^{\infty}_t L^2_x \cap L^2_t L^{6, 2}_x} \leq C \Big(\|Z(0)\|_{L^2} + \|F\|_{L^1_t L^2_x + L^2_t L^{6/5, 2}_x} + \|P_p Z\|_{L^2_t L^{6, 2}_x}\Big).
\end{aligned}\ee
If $V \in L^{3/2, 1}$, then
\be
\|P_c Z\|_{L^1_t L^{6, \infty}_x} \leq C \Big(\|Z(0)\|_{L^{6/5, 1}} + \|F\|_{L^1_t L^{6/5, 1}_x} + \|P_p Z\|_{L^1_t L^{6, \infty}_x}\Big).
\ee
\end{corollary}
We omit the proofs of Corollary \ref{cor_1.7} and Corollary \ref{cor_1.8}, since they are straightforward applications of Theorem \ref{theorem_13}.

One can find numerous other families satisfying properties P1--P4 of Theorem \ref{theorem_13}. As an example, let $\tilde U(t) = e^{i(\int_0^t \beta(\tau) \dd \tau) |\dl|^s)}$, $0 \leq s < 1/4$. It then suffices to see that $(e^{i|\xi|^s})^{\wedge} \in L^1$ on $\set R^3$.

\subsection{History of the problem}

{\emph{The scalar selfadjoint case.}}
The study of decay estimates for Schr\"{o}dinger's equation with a potential has a long history. In the 1970's, Rauch \cite{rau} and Jensen--Kato \cite{jeka} studied the local decay of solutions to (\ref{1.1}) in weighted $L^2$ spaces, under suitable conditions on the potential $V$ and taking into account threshold eigenvalues or resonances. Global decay of solutions was later established by Journ\'{e}--Soffer--Sogge, \cite{jss}, who proved in $\set R^n$ that
\be
\|e^{it(-\Delta+V)}P_c f\|_{L^{\infty}} \leq C |t|^{-n/2} \|f\|_{L^1},
\ee
when $n \geq 3$, zero is neither an eigenvalue, nor a resonance, and, roughly, $|V(x)| \leq C |x|^{-4-n}$ and $\widehat V \in L^1$.

Keel--Tao \cite{tao} obtained endpoint Stric\-hartz estimates for the free Schr\"{o}\-din\-ger and wave equations and introduced a method for obtaining such endpoint inequalities, which can also be used in more general situations. 

Based on \cite{tao}, Rodnianski--Schlag \cite{rodsch} proved nonendpoint Strichartz estimates for large potentials with $\langle x \rangle^{-2-\epsilon}$ decay. Yajima obtained dispersive estimates for Schr\"{o}edinger's equation, under rapid decay assumptions on the potential, both directly \cite{yaj1} and by proving the boundedness of wave operators first \cite{yaj2}. His estimates \cite{yaj4} also apply in the presence of threshold eigenvalues or resonances.

Goldberg proved, in \cite{gol2}, dispersive estimates for almost-critical potentials and, in \cite{gol}, Strichartz estimates for $L^{3/2}$ --- thus scaling-critical --- potentials. Burq--Planchon--Stalker--Tahvildar-Zadeh, in \cite{bpst1} and \cite{bpst2}, obtained decay and Stric\-hartz estimates for a class of $L^{3/2, \infty}$ potentials (under specific regularity assumptions).

Lorentz spaces are essential in the statement of (\ref{1.46}): Lebesgue spaces will only lead to decay rates of the form $t^{-s}$ for $s \in \set R$, which are never integrable (in particular, $t^{-1}$ is not integrable). Lorentz spaces enable us to replace $t^{-1}$ by an integrable rate of decay, in a scaling-invariant setting.

Estimates (\ref{1.46}) and (\ref{eqn1.12}) are related to the more general ones of Foschi, \cite{foschi}. His result, stated here in a form relevant for comparison, is that if
\be\lb{desc}
\|e^{-it \mc H} P_c e^{is \mc H^*} P_c^*\|_{\mc L(L^1, L^{\infty})} \leq C |t-s|^{3/2},
\ee
then
\be
\|e^{-it \mc H} P_c e^{is \mc H^*} P_c^* F\|_{L^q_t L^r_x} \leq C \|F\|_{L^{\tilde q'}_t L^{\tilde r'}_x}.
\ee
$\tilde q'$ and $\tilde r'$ are dual exponents to $\tilde q$ and $\tilde r$ and both $(q, r)$ and $(\tilde q, \tilde r)$ must satisfy
\be
1 \leq q, r \leq \infty,\ \frac 1 q < \frac 3 2- \frac 3 r,\ q \leq r,
\ee
and likewise for $(\tilde q, \tilde r)$; in addition,
\be
\frac 1 q + \frac 1 {\tilde q} = \frac 3 2 \Big(1 - \frac 1 r - \frac 1 {\tilde r}\Big),\ \tilde r < 3r,\ r < 3 \tilde r.
\ee

Then, (\ref{1.46}) is in a forbidden region under Foschi's conditions, being given by $(q, r) = (1, 6)$, $(\tilde q, \tilde r) = (\infty, 6)$. (\ref{eqn1.12}) is the allowed endpoint case $(q, r) = (6, 6)$, $(\tilde q, \tilde r) = (6/5, 6)$; however, (\ref{desc}) is not generally true if $V$ is only in $L^{3/2}$, so (\ref{eqn1.12}) does not follow from \cite{foschi} either.

The result of Foschi differs even more from that of the current paper in the nonselfadjoint case, where $\mc H \ne \mc H^*$, but such differences can be overcome; see, for example, \cite{bec}.

{\emph{The matrix nonselfadjoint case.}} The Hamiltonian (\ref{3.2}) is nonselfadjoint, leading to specific difficulties not present in the selfadjoint case. In particular, one needs to reprove the boundedness of the wave operators, the limiting absorption principle, or even the $L^2$ boundedness of the time evolution. They do not follow as in the selfadjoint case, where, for example, the unitarity of the time evolution immediately implies the $L^2$ boundedness.

In \cite{schlag}, Schlag proved $L^1 \to L^{\infty}$ dispersive estimates for the Schr\"{o}dinger equation with a nonselfadjoint Hamiltonian, as well as non-endpoint Stric\-hartz estimates. Erdo\^{g}an--Schlag \cite{erdsch2} proved $L^2$ bounds for the evolution and a more detailed analysis of decay in the presence of threshold eigenvalues or resonances. In \cite{bec}, endpoint Strichartz estimates in the nonselfadjoint case were obtained following the method of Keel--Tao. More recently, Cuccagna--Mizumatchi \cite{cucmiz} derived all of the above from the boundedness of the wave operators, under more stringent conditions on the potential~$V$.

{\emph{Time-dependent estimates.}} Many estimates concerning the linear Schr\"{o}\-din\-ger equation with time-dependent potentials refer to the time-periodic case. These include the work of Goldberg \cite{gol}, of Costin--Lebowitz--Tanveer \cite{clt}, of Galtbayar--Jensen--Yajima \cite{gjy}, of Bourgain \cite{bou2} (for quasi\-periodic potentials), and of Wang \cite{wang}.

Other results in the time-dependent setting belong to Howland \cite{how}, Kitada--Yajima \cite{kiya}, Rodnianski--Schlag \cite{rodsch}, to Bourgain \cite{bou1}, \cite{bou3}, \cite{bou4}, and to Delort; see \cite{delort}.

The current paper's result, Theorem \ref{theorem_13}, allows only for a specific kind of time dependence of the potential: it can change its position, but not its profile. However, this case is important because it exactly describes the motion of a soliton that arises as the solution of a semilinear dispersive equation. The linear estimates are then useful in controlling a solution that is close to a soliton, hence in proving the latter's stability or instability under small perturbations. The same considerations apply to other nonlinear phenomena, such as vortices.

An even more general situation, studied in this paper, is one where the potential can not only move, but also rotate. Estimates in this case are useful in studying the stability of a soliton that is not radially symmetric and is free to rotate.

Concretely, making use of the techniques introduced here, recent papers such as \cite{bec2}, \cite{nakschl}, or \cite{nakschl2} improve upon older results in \cite{bec} or \cite{schlag} --- by allowing perturbations in energy-space or better, without decay --- and upon \cite{cucmiz}, by allowing the soliton and its perturbations to have a translation movement.


\subsection{Paper outline}

{\emph{Time-independent potentials.}}
We use Wiener's Theorem in the form of Theorem \ref{thm_1.6}, as follows. Consider a decomposition of the potential $V$ into
\be
V = V_1 V_2,\ V_1 = |V|^{1/2} \sgn V,\ V_2 = |V|^{1/2}.
\lb{1.35}
\ee
In the matrix nonselfadjoint case (\ref{3.2}), an analogous decomposition is
\be
V = V_1 V_2,\ V_1 = \sigma_3 \bpm W_1 & W_2 \\ \ov W_2 & W_1 \epm^{1/2},\ V_2 = \bpm W_1 & W_2 \\ \ov W_2 & W_1 \epm^{1/2},\ \sigma_3 = \bpm 1 & 0 \\ 0 & -1 \epm.
\lb{1.36}
\ee
By Duhamel's formula,
\be\begin{aligned}
Z(t) &= e^{it\mc H} Z(0) - i \int_0^t e^{i(t-s) \mc H} F(s) \dd s \\
&= e^{it\mc H_0} Z(0) - i \int_0^t e^{i(t-s) \mc H_0} F(s) \dd s + i \int_0^t e^{i(t-s)\mc H_0} V Z(s) \dd s.
\end{aligned}\ee
In addition, for any multiplicative decomposition $V = V_1 V_2$ of the potential,
\be
V_2 Z(t) = V_2 \bigg(e^{it\mc H_0} Z(0) - i \int_0^t e^{i(t-s) \mc H_0} F(s) \dd s \bigg) + i \int_0^t (V_2 e^{i(t-s)\mc H_0} V_1) V_2 Z(s) \dd s.
\ee
Consider the kernel defined by
\be\lb{1.40}
(T_{V_2, V_1} F)(t)= \int_0^t (V_2 e^{i(t-s)\mc H_0} V_1) F(s) \dd s
\ee
and its Fourier transform in regard to time
\be
\widehat T_{V_2, V_1}(\lambda) = i V_2 R_0(\lambda) V_1.
\ee
To apply Theorem \ref{thm_1.6}, we need to establish that the kernel of $T_{V_2, V_1}$ is time integrable. We do this by Proposition \ref{prop_21}, which implies that, 
for $H=L^2$ and $V_1$, $V_2 \in L^{3, 2}$ as given by (\ref{1.35}) and (\ref{1.36}), $T_{V_2, V_1} \in K = \mc L(H, M_t H)$.

Invertibility of $I-iT_{V_2, V_1}$ within $K$, addressed by Theorem \ref{thm_1.6}, is directly related to Strichartz estimates. Indeed, at least formally we can write
\be\begin{aligned}
V_2 Z(t) &= (I- i T_{V_2, V_1})^{-1} V_2 \bigg(e^{it\mc H_0} Z(0) - i \int_0^t e^{i(t-s) \mc H_0} F(s) \dd s \bigg).
\end{aligned}\ee

If the operator 
$I - iT_{V_2, V_1}$ can be inverted, then the computation is justified. However, as we shall see in Lemma \ref{lem_32}, this is not the case and we need to introduce a correction, to account for the eigenvalues.

Accordingly, in Lemma \ref{lem_32} we construct $\tilde V_1$ and $\tilde V_2$, which equal $V_1$ and $V_2$ plus exponentially small terms, such that $I - iT_{\tilde V_2, \tilde V_1}$ is invertible, then carry on the demonstration for this operator instead.

{\emph{Time-dependent potentials.}} The results in this case are derived in a perturbative manner from those for time-independent potentials. However, the perturbation is not small in Strichartz or similar spaces, so we use a particular method to account for~it.

Denote
\be\begin{aligned}
\tilde T_{\tilde V_2, \tilde V_1} F(t) &= \int_{-\infty}^t \tilde V_2 e^{i(t-s) {\mc H}_0} U(t) U(s)^{-1} \tilde V_1 F(s) \dd s, 
\end{aligned}\ee
respectively
\be\begin{aligned}
\tilde T_{\tilde V_2, I} F(t) &= \int_{-\infty}^t \tilde V_2 e^{i(t-s) {\mc H}_0} U(t) U(s)^{-1} F(s) \dd s.
\end{aligned}\ee
Following several transformations, we reduce equation (\ref{3.157}) to
\be\begin{aligned}
(I - i \tilde T_{\tilde V_2, \tilde V_1}) \tilde V_2 \tilde Z(t) &= \tilde T_{\tilde V_2, I} (-i \tilde F(s) + \delta_{s=0} \tilde Z(0)).
\end{aligned}\ee
If we can invert $I - i \tilde T_{\tilde V_2, \tilde V_1}$, this will imply the desired estimates. To do this, we show that
\be
\lim_{\|\pi'\|_{\infty} \to 0} \|\tilde T_{\tilde V_2, \tilde V_1} - T_{\tilde V_2, \tilde V_1}\| \to 0
\ee
where $T_{\tilde V_2, \tilde V_1}$ has the form (\ref{1.40}). Then, we use the Strichartz estimates obtained in the time-independent case to show that $I - i T_{\tilde V_2, \tilde V_1}$ is invertible.

Indeed, if $A$ is an invertible operator and $\|A-B\| < 1/\|A^{-1}\|$, then $B$ is also invertible and its inverse is given by
\be
B^{-1} = A^{-1} + A^{-1}(A-B)A^{-1} + \ldots.
\ee
This is the case here, so $I - i \tilde T_{\tilde V_2, \tilde V_1}$ is invertible.

\section{Proof of the results}



\subsection{Wiener's Theorem} Let $H$ be a Hilbert space, $\mc L(H, H)$ be the space of bounded linear operators from $H$ to itself, and $M_t H$ be the set of $H$-valued measures of finite mass on the Borel algebra of $\set R$. $M_t H$ is a Banach space, with the norm
\be\lb{3.11}
\|\mu\|_{M_t H} = \sup \bigg\{\sum_{k=1}^n \|\mu(D_k)\|_H \mid D_k \text{ disjoint Borel sets}\bigg\}.
\ee
Note that the absolute value of $\mu \in M_t H$ given by
\be
|\mu|(D) = \sup\bigg\{\sum_{k=1}^n \|\mu(D_k)\|_H \mid \bigcup_{k=1}^n D_k = D,\ D_k\text{ disjoint Borel sets}\bigg\}
\ee
is a positive measure of finite mass (bounded variation) and $\|\mu(D)\|_H \leq |\mu|(D)$. By the Radon--Nikodym Theorem, $\mu$ is in $M_t H$ if and only if it has a decomposition
\be
\mu = \mu_{\infty} |\mu|
\ee
with $|\mu| \in M_t$ (the space of real-valued measures in $t$ of finite mass) and $\mu_{\infty}~\in L^{\infty}_{d|\mu|(t)} H$.

Furthermore, $\|\mu\|_{M_t H} = \||\mu|\|_{M}$ and the same holds if we replace $H$ by any Banach space.

\begin{definition}\lb{def_k} Let $K = \mc L(H, M_t H)$ be the algebra (under convolution) of bounded operators from $H$ to $M_t H$.
\end{definition}

\begin{observation}
Examples of elements of $K$ are:
\begin{list}{\labelitemi}{\leftmargin=1em}
\item the identity $I$: $I(h) = \delta_{t=0} h$ for any $h \in H$;
\item scalar functions or measures, in general: if $f \in L^1_t$ or $M_t$ and $h \in H$, let $f(h) = f(t) h$;
\item product form operators $k(t)=f(t) A_0$, where $f \in M_t$, $A_0 \in \mc L(H, H)$: for $h \in H$, let $A(t) h = f(t) (A_0 h)$;
\item collections of operators $A(t)$, such that $A(t) \in \mc L(H, H)$ for almost all $t$ and $\int \|A(t)\|_{\mc L(H, H)} \dd t < \infty$;
\item finally, we also give an example that falls under none of the above categories, but is covered by the definition of $K$. Let $H=L^2(\set R)$, $f_0 \in L^2(\set R)$ be fixed, and $A(t) = \big(f(t) f_0(t)\big) f_0$. Then, for any $f \in L^2$
\be
\int_{\set R} \|A(t) f\|_2 \dd t \leq \|f_0\|_2 \int_{\set R} |f(t) f_0(t)| \dd t \leq \|f_0\|_2^2 \|f\|_2,
\ee
so $A \in K$, but $A(t)$ is not bounded on $L^2$ for any $t$, because $L^2 \not \subset L^{\infty}$.
\end{list}
\end{observation}

The following lemma also serves to clarify Definition \ref{def_k}:
\begin{lemma}
$K$ takes $M_t H$ into itself by convolution, is a Banach algebra under convolution, and multiplication by bounded continuous functions (and $L^{\infty}$ Borel measurable functions) is bounded on $K$:
\be
\|f A\|_K \leq \|f\|_{\infty} \|A\|_K.
\lb{3.8}
\ee
Furthermore, by integrating an element $A$ of $K$ over $\set R$ one obtains $\int_{\set R} A \in \mc L(H, H)$, with $\|\int_{\set R} A\|_{\mc L(H, H)} \leq \|k\|_K$.
\end{lemma}
\begin{proof} Boundedness of multiplication by continuous or $L^{\infty}$ functions follows from the decomposition $\mu = \mu_0 |\mu|$ for $\mu \in M_t H$. The last stated property is a trivial consequence of the definition of $M_t H$.

Let $\mu \in M_t H$, $A \in K$. Consider the product measure $\tilde \mu$ first defined on product sets $D_1 \times D_2 \subset \set R \times \set R$ by $\tilde \mu(D_1 \times D_2) = k(\mu(D_2))(D_1)$. This is again a measure of finite mass, $\tilde \mu \in M_{t, s} H$, and
\be
\|\tilde \mu\|_{M_{t, s} H} \leq \|A\|_K \|\mu\|_{M_t H}.
\ee
We then naturally define the convolution of an element of $K$ with an element of $M_t H$, by setting $A(\mu)(D)=\tilde \mu(\{(t, s) \mid t+s \in D\})$.

Thus, each $A \in K$ defines a bounded translation-invariant linear map from $M_t H$ to itself:
\be
\|A(\mu)\|_{M_t H} \leq \|A\|_K \|\mu\|_{M_t H}.
\lb{3.9}
\ee
The correspondence is bijective, as any translation-invariant $\tilde A \in \mc L(M_t H, M_t H)$ defines an element $A \in K$ by $A(h) = \tilde A(\delta_{t=0}h)$. These operations are indeed inverses of one another.

Associativity follows from Fubini's Theorem. $K$ is a Banach space by definition. The algebra property of the norm is immediate from (\ref{3.9}).
\end{proof}

Note that, due to our choice of a Hilbert space $H$, if $A \in K$ then $A^* \in K$ as well.

Define the Fourier transform of any element in $K$ by
\be
\widehat A(\lambda) = \int_{\set R} e^{-it\lambda} \dd A(t).
\ee
This is a bounded operator from $H$ to itself.
By dominated convergence, $\widehat A(\lambda)$ is a strongly continuous (in $\lambda$) family of operators for each $A$ and, for each $\lambda$,
\be
\|\widehat A(\lambda)\|_{H \to H} \leq \|A\|_K.
\ee
This follows from (\ref{3.8}).

The Fourier transform on $\set R$ of the identity is $\widehat I(\lambda) = I$ for every $\lambda$; $\widehat{A^*} = (\widehat A)^*$. Also, the Fourier transform takes convolution to composition.

If a kernel $A \in K$ has both a left and a right inverse, they must be the same, $b = b * I = b * (A * B) = (b * A) * B = I * B = B$.

Fix a cutoff function $\chi$ supported on a compact set and which equals one on some neighborhood of zero. We specify that the inverse Fourier transform on $\set R$ is
\be
f^{\vee}(t) = \frac 1 {2\pi} \int_{\set R} e^{it\lambda} f(\lambda) \dd \lambda.
\ee
\begin{theorem}\lb{thm7}
Let $K$ be the operator algebra of Definition \ref{def_k}. If $A \in K$ is invertible then $\widehat A(\lambda)$ is invertible for every $\lambda$. Conversely, assume $\widehat A(\lambda)$ is invertible for each $\lambda$, $A = I + L$, and
\be
\lim_{\epsilon \to 0} \|L(\cdot + \epsilon) - L\|_K = 0,\ \lim_{R \to \infty}\|(1-\chi(t/R)) L(t)\|_K =0.
\ee
Then $A$ is invertible.
\end{theorem}
\begin{observation}
If $L$ is in a closed unital subalgebra of $K$ with some specific properties, then its inverse will also belong to the same. This stability property helps in proving a precise decay rate for the inverse.

For example, if $\|A(t)\| \leq C \langle t \rangle^{-3/2}$, then $\|A^{-1}(t)\| \leq \tilde C \langle t \rangle^{-3/2}$ as well.
\end{observation}

We do not use compactness or Fredholm's alternative explicitly. A subset of $L^1_t$ is precompact if and only if its elements are equicontinuous and decay uniformly at infinity --- conditions that are actually required above, but not equivalent to compactness on $L^1_t H$.

The set $K_c$ of equicontinuous operators, that is
\be
K_c = \{L \in K \mid \lim _{\epsilon \to 0} \|L(\cdot + \epsilon) - L\|_K = 0\}
\ee
is a closed ideal, is translation invariant, contains the set of those kernels which are strongly measurable and $L^1$ (but $K_c$ is strictly larger), and $I$ is not in it.

Likewise, the set $K_0$ of kernels that decay at infinity, that is
\be
K_0 = \{L \in K \mid \lim_{R \to \infty}\|\chi_{|t|>R} L(t)\|_K =0\},
\ee
is a closed subalgebra. It contains the strong algebras that we defined above. Note that for operators $A \in K_0$ the Fourier transform is also a norm-continuous family of operators, not only strongly continuous.

The proof will also ensure that, if $I+L$ belongs to $(K_c \cap K_0) \oplus \set C I$ and is invertible in $K$, then its inverse is of the same~form.
\begin{proof}[Proof of Theorem \ref{thm7}]
If $A$ is invertible, that is $A * A^{-1} = A^{-1} * A = I$, then applying the Fourier transform yields
\be
\widehat A (\lambda) \widehat {A^{-1}}(\lambda) = \widehat {A^{-1}}(\lambda) \widehat A (\lambda) = I
\ee
for each $\lambda$, so $\widehat A(\lambda)$ is invertible.

Conversely, assume $\widehat A(\lambda)$ is invertible for every $\lambda$. Without loss of generality, we can replace $A$ with $A * A^*$. Then, the theorem's hypotheses are preserved and, in addition, $\widehat A(\lambda)$ is  non-negative for every $\lambda$.

A non-negative operator is invertible if and only if it is strictly positive, so $\widehat A(\lambda) >0$ for every $\lambda$. 

Fix $\lambda_0 \in \set R$ and let $\chi$ be a Schwartz-class function, such that $\supp \chi \subset [-2, 2]$ and $\chi(\lambda) = 1$ for $t \in [-1, 1]$. 
Also let $\eta = \widehat \chi$, $\eta_{\epsilon}(t) = \nobreak \epsilon \eta(\epsilon t) = \widehat{\chi(\epsilon^{-1} \cdot)}$.

For any kernel $B \in K_0$ --- that decays at infinity --- we next show that
\be\lb{deca}
\big\|\big((e^{is\lambda_0}\eta_{\epsilon}(s)) * B\big)(t) - e^{it\lambda_0} \eta_{\epsilon}(t) \widehat B(\lambda_0) \big\|_K \to 0.
\ee
On one hand, outside a large radius,
\be
\lim_{R \to \infty} \|(1-\chi(t/R)) B(t)\|_K = 0,
\ee
while on a fixed ball of radius $R$, as $\epsilon \to 0$,
\be
\bigg\|\int_{|s|\leq R} \big(\eta_{\epsilon}(t) - \eta_{\epsilon}(t-s)\big) B(s) \dd s\bigg\|_K \leq \|B\|_K \cdot \Big\|\sup_{|s|\leq R}\big(\eta_{\epsilon}(t) - \eta_{\epsilon}(t - s)\big)\Big\|_{L^1_t} \to 0.
\ee
By taking $R \to \infty$ and integrating separately inside and outside this radius, we evaluate the following expression in $t$:
\be\begin{aligned}\lb{2.18}
&\bigg\|\big((e^{is\lambda_0}\eta_{\epsilon}(s)) * B\big)(t) - e^{it\lambda_0} \eta_{\epsilon}(t) \int_{\set R} e^{-is\lambda_0} B(s) ds\bigg\|_K \leq \\
&\leq \bigg\|e^{it\lambda_0} \int \big(\eta_{\epsilon}(t) - \eta_{\epsilon}(t-s)\big) \chi(s/R) B(s) \dd s \bigg\|_K + 2 \|\eta_{\epsilon}\|_1 \big\|(1-\chi(t/R)) B(t)\big\|_K \\
&\leq \bigg\|\int_{|s|\leq 2R} \big(\eta_{\epsilon}(t) - \eta_{\epsilon}(t-s)\big) B(s) \dd s\bigg\|_K + 2 \|\eta_{\epsilon}\|_1 \big\|(1-\chi(t/R)) B(t)\big\|_K \to 0.
\end{aligned}\ee
In (\ref{2.18}) we employ the fact that $\|\eta_{\epsilon}\|_1 = \|\eta\|_1$ is independent of $\epsilon$.


Let
\be\begin{aligned}
A_{\epsilon}(t) &= (I - e^{it\lambda_0} \eta_{\epsilon}(t)) + (e^{i t \lambda_0} \eta_{\epsilon}(t)) * A,\\
\tilde A_{\epsilon}(t) &= \big(I - e^{it\lambda_0} \eta_{\epsilon}(t)\big) + e^{it\lambda_0} \eta_{\epsilon}(t) \widehat A(\lambda_0).
\end{aligned}\ee
Note that $\widehat {A_{\epsilon}}(\lambda) = \widehat A(\lambda)$ when $|\lambda-\lambda_0| \leq \epsilon$: 
\be
\widehat {A_{\epsilon}}(\lambda) = \big(1-\chi(\epsilon^{-1} (\lambda-\lambda_0))\big)I + \chi(\epsilon^{-1} (\lambda-\lambda_0)) \widehat A(\lambda).
\ee
By virtue of (\ref{deca}),
\be\lb{2.22}
\lim_{\epsilon \to 0} \|A_{\epsilon}- \tilde A_{\epsilon}\|_K = 0.
\ee
The Fourier transform of $\tilde A_{\epsilon}$ has the form
\be
\tilde A^{\wedge}_{\epsilon}(\lambda) = \big(1 - \chi(\epsilon^{-1} (\lambda-\lambda_0))\big) I + \chi(\epsilon^{-1} (\lambda-\lambda_0)) \widehat A(\lambda_0).
\ee
Since $\widehat A(\lambda_0) > 0$ is a strictly positive operator, $\tilde A^{\wedge}_{\epsilon}(\lambda)$ is also a positive operator for every $\lambda$, so it can be inverted for every $\lambda$:
\be
\|\tilde A^{\wedge}_{\epsilon}(\lambda)^{-1}\|_{\mc L(H, H)} \leq \max (2, 2 \|\widehat A(\lambda_0)^{-1}\|_{\mc L(H, H)}).
\ee
We can easily differentiate $\tilde A^{\wedge}_{\epsilon}(\lambda)$ as a function of $\lambda$, taking values in $\mc L(H, H)$:
\be
\partial_{\lambda} \tilde A^{\wedge}_{\epsilon} (\lambda) = \epsilon^{-1} \chi'(\epsilon^{-1}\lambda) (\widehat A(\lambda_0) - I).
\ee
Thus, we can do the same for its inverse:
\be
\partial_{\lambda} (\tilde A^{\wedge}_{\epsilon})^{-1} (\lambda) = - \tilde A^{\wedge}_{\epsilon}(\lambda)^{-1} \epsilon^{-1} \chi'(\epsilon^{-1}\lambda) (\widehat A(\lambda_0) - I) \tilde A^{\wedge}_{\epsilon}(\lambda)^{-1}.
\ee
Repeating this an arbitrary number of times, we obtain that the inverse is infinitely differentiable in $\lambda$, in a strong sense.

Moreover, $\tilde A^{\wedge}_{\epsilon}(\lambda) - I$ is compactly supported in $\lambda$, so the same is true for the inverse: $\tilde A^{\wedge}_{\epsilon}(\lambda)^{-1} - I$ is compactly supported.

Since $\tilde A^{\wedge}_{\epsilon}(\lambda)^{-1} - I$ is compactly supported and infinitely differentiable, it follows that its Fourier transform is in $L^1$. Therefore, $(\tilde A_{\epsilon})^{-1} \in K$.

Because under rescaling $\tilde A_{\epsilon}(t) = \epsilon \tilde A_1(\epsilon t)$, the same holds for their inverses, so $\|\tilde A_{\epsilon}^{-1}\|_K$ is independent of $\epsilon$. By (\ref{2.22}), it follows that $A_{\epsilon}$ is also invertible for sufficiently small $\epsilon$.

We have to consider infinity separately. As above, let $\eta_R = R \chi^{\vee}(R \cdot)$ and
\be
A_R = (I - \eta_R) * A + \eta_R,
\ee
so that $\widehat A_R(\lambda) = \widehat A(\lambda)$ when $|\lambda|>2R$:
\be
\widehat A_R(\lambda) = (1-\chi(R^{-1} \lambda)) \widehat A(\lambda) + \chi(R^{-1} \lambda) I.
\ee
Then
\be
A_R - I = (I - \eta_R) * (A - I) = (I - \eta_R) * L.
\ee
At this step we use the equicontinuity assumption of the hypothesis, namely
\be
\lim_{\epsilon \to 0}\|L-L(\cdot + \epsilon)\|_K = 0.
\ee
We write $\eta_R = (\chi_{[-\epsilon, \epsilon]} \eta_R) + (1-\chi_{[-\epsilon, \epsilon]}) \eta_R$, where $\chi_{[-\epsilon, \epsilon]}$ is the characteristic function of $[-\epsilon, \epsilon]$. As $\epsilon \to 0$,
\be\begin{aligned}
&\|(\chi_{[-\epsilon, \epsilon]} \eta_R) * L - \big(\textstyle\int_{\set R} \chi_{[-\epsilon, \epsilon]}(s) \eta_R(s) \dd s \big) L\|_K \leq \\
&\leq \Big(\int_{\set R} \chi_{[-\epsilon, \epsilon]}(s) |\eta_R(s)| \dd s \Big) \sup_{\delta \leq \epsilon} \|L-L(\cdot + \delta)\|_K \to 0,
\end{aligned}\ee
because $\int_{\set R} |\eta_R(s)| \dd s$ is independent of $R$. At the same time,
\be\begin{aligned}
&\lim_{R \to \infty} \|(1-\chi_{[-\epsilon, \epsilon]}) \eta_R\|_1 \to 0,
\end{aligned}\ee
so we obtain
\be
\lim_{R \to \infty} \|(I - \eta_R) * L\|_K = 0.
\ee
Thus, we can invert $A_R$ for large $R$. It follows that on some neighborhood of infinity the Fourier transform of $A$ equals that of an invertible operator.

Finally, 
consider a finite open cover of $\set R$ of the form
\be
\set R = D_{\infty} \cup \bigcup_{j=1}^n D_j,
\ee
where $D_j$ are open sets and $D_{\infty}$ is an open neighborhood of infinity, such that for $1 \leq j \leq n$ and for $j=\infty$ we have $\widehat A^{-1} = \widehat A_j^{-1}$ on the open set $D_j$. Take a smooth partition of unity subordinate to this cover, that is
\be
1 = \sum_j \chi_j,\ \supp \chi_j \subset D_j, \chi_j \text{ smooth}.
\ee
Then the following element of $K$ is the inverse of $A$:
\be
A^{-1} = \sum_{j=1}^n \widehat \chi_j * A_j^{-1} + (I - \sum_{j=1}^n \widehat \chi_j) * A_{\infty}^{-1}.
\ee
\end{proof}

We are also interested in whether, if $A$ is upper triangular, meaning that $A$ is supported on $[0, \infty)$ in the $t$ variable, the inverse of $A$ has the same property.
\begin{lemma}\lb{lemma8}
Given $A \in K$ upper triangular with $A^{-1} \in K$, $A^{-1}$ is upper triangular if and only if $\widehat A$ can be extended to a weakly analytic family of invertible operators in the lower half-plane, which is continuous up to the boundary, uniformly bounded, and with uniformly bounded inverse.
\end{lemma}
Intuitively, this lemma shows that convolution operators in $K$ with Volter\-ra-type kernels (i.e.\ no future dependence) have inverses with the same property, if the inverses exist.
\begin{proof}
In one direction, assume that both $A^{-1}$ and $A$ are upper triangular. Then, one can construct $\widehat A(\lambda)$ and $\widehat {A^{-1}}(\lambda)$ in the lower half-plane, as their defining integrals converge there. Strong continuity follows by dominated convergence and weak analyticity by means of the Cauchy integral formula. Furthermore, both $\widehat A(\lambda)$ and $\widehat {A^{-1}}(\lambda)$ are bounded, by $\|A\|_K$ and $\|A^{-1}\|_K$ respectively, and they are inverses of one another.

Conversely, let $A_- = \chi_{(-\infty, 0]} A^{-1}$ and $A_+ = \chi_{[0, \infty)} A^{-1}$. Taking the Fourier transform, one has that for each $\lambda$
\be
\widehat {A_-}(\lambda) + \widehat {A_+}(\lambda) = \widehat {A^{-1}}(\lambda).
\ee
On the lower half-plane, $\widehat {A^{-1}}(\lambda) = (\widehat A)^{-1}(\lambda)$ is uniformly bounded by hypothesis. Likewise, $\widehat {A_+}(\lambda)$ is bounded as the Fourier transform of an upper triangular operator. Thus, $\widehat {A_-}(\lambda)$ too is bounded on the lower half-plane.

However, $A_-$ is lower triangular, so its Fourier transform is also bounded in the upper half-plane. It follows that $\widehat {A_-}$ is a holomorphic function, bounded on the whole plane. By Liouville's theorem, then, $\widehat {A_-}$ must be constant, so $A_-$ can only have singular support at zero. Therefore $A$ is upper triangular.
\end{proof}

If $V \in L^{3/2}$, then we need an analogue of Wiener's Theorem for neither $\mc L(L^2, L^2) = \widehat {L^{\infty}}$ nor $\mc L(L^1, L^1)  = \mc M$, but an interpolation space of the two.
\begin{theorem}\lb{thm6/5}
Let $A \in K_{6/5} := (K, \widehat {L^{\infty}}_{\lambda} H)_{[1/3]}$. Assume $\widehat A(\lambda)$ is invertible for each $\lambda$, $A = I + L$, and
\be
\lim_{\epsilon \to 0} \|L(\cdot + \epsilon) - L\|_{K_{6/5}} = 0,\ \lim_{R \to \infty}\|(1-\chi(t/R)) L(t)\|_{K_{6/5}} =0.
\ee
Then $A$ is invertible in $K_{6/5}$.
\end{theorem}
\begin{proof} Note that elements of $K$ give rise to bounded operators on $L^1_t H$, while elements of $\widehat {L^{\infty}}_{\lambda} H$ do so for $L^2_t H$, so by interpolation elements of $K_{6/5}$ are bounded on $L^{6/5}_t H$.

Furthermore, both $L^1_t H$ and $\widehat {L^{\infty}}_{\lambda} H$ are translation-invariant algebras and their elements have Fourier transforms bounded in $\mc L(H, H)$; thus, $K_{6/5}$ also possesses the same properties.

Then, the proof is identical to that of Wiener's Theorem, Theorem \ref{thm_1.6}, earlier in this section.
\end{proof}

We next apply this abstract theory to the particular case of interest.

\subsection{The free evolution and resolvent in three dimensions}
We return to the concrete case (\ref{eq_3.2}) or (\ref{3.2}) of a linear Schr\"{o}dinger equation  on $\set R^3$ with scalar or matrix nonselfadjoint potential $V$. For simplicity, the entire subsequent discussion revolves around the case of three spatial dimensions.


In order to apply the abstract Wiener theorem, Theorem \ref{thm7}, it is necessary to exhibit an operator-valued measure of finite mass. Accordingly, we start by proving Proposition \ref{prop_21}.

\begin{proof}[Proof of Proposition \ref{prop_21}] We provide two proofs --- a shorter one based on real interpolation and a longer one, using the atomic decomposition of Lorentz spaces (Lemma \ref{lemma_30}), that exposes the proof machinery underneath.

Without loss of generality, we consider only the interval $[0, \infty)$.

Following the first approach, note that by duality (\ref{eq_3.39}) is equivalent to
\be\lb{eq_3.40}
\sum_{n \in \set Z} 2^n \int_{2^n}^{2^{n+1}} |\langle e^{it\mc H_0} f, g(t) \rangle| \dd t \leq C \|f\|_{L^{6/5, 1}} \|g\|_{L^1_t L^{6/5, 1}_x}
\ee
holding for any $f \in L^{6/5, 1}$ and $g \in L^1_t L^{6/5, 1}_x$.

From the usual dispersive estimate
\be
\|e^{-it\Delta} f\|_{p'} \leq t^{3/2(1-2/p)} \|f\|_p
\ee
we obtain that
\be
\int_{2^n}^{2^{n+1}} |\langle e^{it\mc H_0} f, g(t) \rangle| \dd t \leq 2^{n(3/2-3/p)} \|f\|_p \|g\|_{L^1_t L^p_x}.
\ee
Restated, this means that the bilinear mapping
\be\begin{aligned}
& T:L^p \times L^1_t L^p_x \to \ell^{\infty}_{3/p-3/2} \text{ (following the notation of Proposition \ref{prop_33})},\\
& T=(T_n)_{n \in \set Z},\ T_n(f, g) = \int_{2^n}^{2^{n+1}} \langle e^{it\mc H_0} f, g(t) \rangle \dd t
\end{aligned}\ee
is bounded for $1 \leq p \leq 2$.

Interpolating between $p=1$ and $p=2$, by using the real interpolation method (Theorem \ref{thm_32}) with $\theta=1/3$ and $q_1=q_2=1$, directly shows that
\be
T: (L^1, L^2)_{1/3, 1} \times (L^1_t L^1_x, L^1_t L^2_x)_{1/3, 1} \to (\ell^{\infty}_{3/2}, \ell^{\infty}_0)_{1/3, 1}
\ee
is bounded. By Proposition \ref{prop_33},
\be\begin{aligned}
(L^1, L^2)_{1/3, 1} &= L^{6/5, 1},\\
(L^1_t L^1_x, L^1_t L^2_x)_{1/3, 1} &= L^1_t (L^1_x, L^2_x)_{1/3, 1} = L^1_t L^{6/5, 1}_x,\\
(\ell^{\infty}_{3/2}, \ell^{\infty}_0)_{1/3, 1} &= \ell^1_1.
\end{aligned}
\ee
Hence $T$ is bounded from $L^{6/5, 1} \times L^1_t L^{6/5, 1}_x$ to $\ell^1_1$, which implies (\ref{eq_3.39}).

The alternative approach is based on the atomic decomposition of $L^{6/5, 1}$. By Lemma \ref{lemma_30},
\be\begin{aligned}
f = \sum_{j \in \set Z} \alpha_j a_j, g(t) = \sum_{k \in \set Z} \beta_k(t) b_k(t),
\end{aligned}\ee
where $a_j$ and, for each $t$, $b_k(t)$ are atoms with
\be\begin{aligned}
\mu(\supp(a_j)) &= 2^j,& \esssup |a_j| &= 2^{-5j/6},\\
\mu(\supp(b_k(t))) &= 2^k,& \esssup |b_k(t)| &= 2^{-5k/6}
\end{aligned}\ee
(here $\mu$ is the Lebesgue measure on $\set R^3$), and the coefficients $\alpha_j$ and $\beta_k(t)$ satisfy
\be
\sum_{j \in \set Z} |\alpha_j| \leq C \|f\|_{L^{6/5, 1}},\ \sum_{k \in \set Z} |\beta_k(t)| \leq C \|g(t)\|_{L^{6/5, 1}}.
\ee
Integrating in time and exchanging summation and integration lead to
\be
\sum_k \int_0^{\infty} |\beta_k(t)| \dd t \leq C \|g\|_{L^1_t L^{6/5, 1}_x}.
\ee
Since (\ref{eq_3.40}) is bilinear in $f$ and $g$, it suffices to prove it for only one pair of atoms. Fix indices $j_0$ and $k_0 \in \set Z$; the problem reduces to showing that
\be\lb{eq_3.51}
\sum_{n \in \set Z} 2^n \int_{2^n}^{2^{n+1}} \langle e^{it \mc H_0} a_{j_0}, \beta_{k_0}(t) b_{k_0}(t) \rangle \dd t \leq C \int_0^{\infty} |b_{k_0}(t)| \dd t.
\ee
The reason for making an atomic decomposition is that atoms are in $L^1 \cap L^{\infty}$, instead of merely in $L^{6/5, 1}$, enabling us to employ both $L^1$ to $L^{\infty}$ decay and $L^2$ boundedness estimates in the study of their behavior. For each $n$,
\be\begin{aligned}\lb{eq_1.352}
\int_{2^n}^{2^{n+1}} \langle e^{it \mc H_0} a_{j_0}, \beta_{k_0}(t) b_{k_0}(t) \rangle \dd t &\leq C 2^{-3n/2} \|a_{j_0}\|_1 \sup_t \|b_{k_0}(t)\|_1 \int_{2^n}^{2^{n+1}} |\beta_{k_0}(t)| \dd t \\
&= C 2^{-3n/2} 2^{j_0/6} 2^{k_0/6} \int_{2^n}^{2^{n+1}} |\beta_{k_0}(t)| \dd t
\end{aligned}\ee
as a consequence of the $L^1 \to L^{\infty}$ decay estimate. At the same time, by the $L^2$ boundedness of the evolution,
\be
\begin{aligned}\lb{eq_1.353}
\int_{2^n}^{2^{n+1}} \langle e^{it \mc H_0} a_{j_0}, \beta_{k_0}(t) b_{k_0}(t) \rangle \dd t & \leq C \|a_{j_0}\|_2 \sup_t \|b_{k_0}(t)\|_2 \int_{2^n}^{2^{n+1}} |\beta_{k_0}(t)| \dd t \\
&= C 2^{-j_0/3} 2^{-k_0/3} \int_{2^n}^{2^{n+1}} |\beta_{k_0}(t)| \dd t.
\end{aligned}
\ee
Using the first estimate (\ref{eq_1.352}) for large $n$, namely $n\geq j_0/3 + k_0/3$, and the second estimate (\ref{eq_1.353}) for small $n$, $n<j_0/3 + k_0/3$, we always obtain that
\be
\int_{2^n}^{2^{n+1}} \langle e^{it \mc H_0} a_{j_0}, \beta_{k_0}(t) b_{k_0}(t) \rangle \dd t \leq C 2^{-n} \int_{2^n}^{2^{n+1}} |\beta_{k_0}(t)| \dd t.
\ee
Multiplying by $2^n$ and summing over $n \in \set Z$ we retrieve (\ref{eq_3.51}), which in turn proves (\ref{eq_3.40}).
\end{proof}

The resolvent of the unperturbed Hamiltonian, $R_0(\lambda) = (\mc H_0 - \lambda)^{-1}$, is given by (\ref{eq_3.55}) in the scalar case (\ref{eq_3.2}) and (\ref{eq_3.56}) in the matrix case (\ref{3.2}). In either case, $R_0(\lambda) = (\mc H_0 - \lambda)^{-1}$ is an analytic function, on $\set C \setminus [0, \infty)$ or respectively on $\set C \setminus ((-\infty, -\mu] \cup [\mu, \infty))$. It can be extended to a continuous function in the closed lower half-plane or the closed upper half-plane, but not both at once, due to a jump discontinuity on the real line.

The resolvent is the Fourier transform of the time evolution. We formally state the known connection between $e^{it\mc H_0}$ and the resolvent $R_0 = (\mc H_0 - \lambda)^{-1}$.
\begin{lemma}\lb{lemma_22}
Let $\mc H_0$ be given by (\ref{eq_3.2}) or (\ref{3.2}). For any $f \in L^{6/5, 1}$ and $\lambda$ in the lower half-plane, the integral
\be
\lim_{\rho \to \infty} \int_0^{\rho} e^{-it\lambda} e^{it \mc H_0} f \dd t
\lb{3.47}
\ee
converges in the $L^{6, \infty}$ norm and equals $i R_0(\lambda) f$ or $i R_0(\lambda-i0) f$ in case $\lambda \in \set R$.

Furthermore, for real $\lambda$,
\be
\lim_{\rho \to \infty} \int_{-\rho}^{\rho} e^{-it\lambda} e^{it \mc H_0} f \dd t = i(R_0(\lambda-i0) - R_0(\lambda+i0)) f,
\lb{3.39}\ee
also in the $L^{6, \infty}$ norm.
\end{lemma}
\begin{proof}

Note that (\ref{3.47}) is dominated by (\ref{eqn_3.39}),
\be
\int_0^{\infty} \|e^{it \mc H_0} f\|_{L^{6, \infty}} \dd t,
\ee
and this ensures its absolute convergence. Next, both (\ref{3.47}), as a consequence of the previous argument, and $i R_0(\lambda+i0)$ are bounded operators from $L^{6/5, 1}$ to $L^{6, \infty}$. To show that they are equal, it suffices to address this issue over a dense set. Observe that
\be
\int_0^{\rho} e^{-it(\lambda-i\epsilon)} e^{it \mc H_0} f \dd t = iR_0(\lambda-i\epsilon) (f - e^{-i\rho(\lambda-i\epsilon)} e^{i\rho \mc H_0} f).
\ee
Thus, if $f \in L^2 \cap L^{6/5, 1}$, considering the fact that $e^{it \mc H_0}$ is unitary and $R_0(\lambda-i\epsilon)$ is bounded on $L^2$,
\be
\lim_{\rho \to \infty} \int_0^{\rho} e^{-it(\lambda-i\epsilon)} e^{it \mc H_0} f \dd t = iR_0(\lambda-i\epsilon) f.
\lb{3.54}
\ee
Letting $\epsilon$ go to zero, the left-hand side in (\ref{3.54}) converges, by dominated convergence, to (\ref{3.47}), while the right-hand side (also by dominated convergence, using the explicit form (\ref{eq_3.55})-(\ref{eq_3.56}) of the operator kernels) converges to $iR_0(\lambda-i0) f$. Statement (\ref{3.39}) follows directly.
\end{proof}

\subsection{The exceptional set and the resolvent}
This section and the next mainly consist in a generalization of known results, following models such as \cite{agmon}, \cite{schlag}, and others cited. The main novelty is doing the proofs in a scaling- and translation-invariant setting.

We explore properties of the perturbed resolvent $R_V$. Important in this context is the \emph{Birman-Schwinger operator},
\be\lb{2.56}
\widehat T_{V_2, V_1} (\lambda) = i V_2 R_0(\lambda) V_1,
\ee
where $V = V_1 V_2$ and $V_1$, $V_2$ are as in (\ref{1.35}) or (\ref{1.36}).

The Birman-Schwinger operator is uniformly bounded on $L^2$ for $\lambda \in \set C$, including boundary values of $\lambda$ along the real line.
\begin{lemma}\lb{le27} Take $V \in L^{3/2, \infty}$; then there exists $C$ such that for any $\lambda \in \set C$
\be
\|V_2 R_0(\lambda) V_1\|_{\mc L(L^2_x, L^2_x)} \leq C < \infty.
\ee
\end{lemma}
\begin{proof} By (\ref{eq_3.55}) or (\ref{eq_3.56}), the convolution kernel of $R_0(\lambda)$ is in $L^{3, \infty}$ with uniformly bounded norm. Likewise, $V_1$ and $V_2$ are in $L^{3, \infty}$. Then
\be
\|V_2 R_0(\lambda) V_1 f\|_{L^2_x} \leq C \|R_0(\lambda) V_1 f\|_{L^{6/5, 2}_x} \leq C \|V_1 f\|_{L^{6, 2}_x} \leq C \|f\|_{L^2_x}.
\ee
Indeed, by Proposition \ref{prop_holder}, $L^{6/5, 2} * L^{3, \infty} \mapsto L^{6, 2}$, $L^2 \cdot L^{3, \infty} \mapsto L^{6/5, 2}$, and $L^{6, 2} \cdot L^{3, \infty} \mapsto L^2$.
\end{proof}

The relation between the Birman-Schwinger operator and the perturbed resolvent $R_V=(\mc H_0 + V - \lambda)^{-1}$ is that
\be\lb{eq_3.64}
R_V(\lambda) = R_0(\lambda) - R_0(\lambda) V_1 (I + V_2 R_0(\lambda) V_1)^{-1} V_2 R_0(\lambda)
\ee
and
\be\lb{eq_3.65}
(I + V_2 R_0(\lambda) V_1)^{-1} = I - V_2 R_V(\lambda) V_1.
\ee
Both follow by direct computation from the resolvent identity:
\be
R_V(\lambda) = R_0(\lambda) - R_0(\lambda) V R_V(\lambda) = R_0(\lambda) - R_V(\lambda) V R_0(\lambda).
\ee




\begin{definition}\lb{def_7} Given $V \in L^{3/2, \infty}_0$, its exceptional set $\mc E$ is the set of $\lambda$ in the complex plane for which $I - i\widehat T_{V_2, V_1}(\lambda)$ is not invertible from $L^2$ to itself.
\end{definition}

Other choices of $V_1$ and $V_2$ such that $V = V_1 V_2$, $V_1$, $V_2 \in L^{3, \infty}$ lead to the same operator up to conjugation.

Below we summarize a number of observations concerning the exceptional sets of operators in the form (\ref{eq_3.2}) or (\ref{3.2}).
\begin{proposition}\lb{prop28} Assume $V \in L^{3/2, \infty}_0$ is a potential as in (\ref{eq_3.2}) or (\ref{3.2}) and denote its exceptional set by $\mc E$.

\begin{itemize}\item[i)] $\mc E$ is bounded and discrete outside $\sigma(\mc H_0)$, but can accumulate toward $\sigma(\mc H_0)$. $\mc E \cap \sigma(\mc H_0)$ has null measure (as a subset of $\set R$). Elements of $\mc E \setminus \sigma(\mc H_0)$ are eigenvalues of $\mc H = \mc H_0 + V$.


\item[ii)] If $V$ is real and matrix-valued as in (\ref{3.2}), then embedded exceptional values must be eigenvalues, except for the endpoints of $\sigma(\mc H_0)$, which need not be eigenvalues. If $V$ is complex matrix-valued as in (\ref{3.2}), there is no restriction on embedded exceptional values.

\item[iii)] If $V$ is complex scalar as in (\ref{eq_3.2}) or complex matrix-valued as in (\ref{3.2}), then $\mc E$ is symmetric with respect to the real axis. In case $V$ is real-valued and as in (\ref{3.2}), $\mc E$ is symmetric with respect to both the real axis and the origin.
\end{itemize}
\end{proposition}

\begin{proof}
By Lemma \ref{le27}, for $V \in L^{3/2, \infty}$, $V_2 R_0(\lambda) V_1$ is $L^2$-bounded for every~$\lambda$.

The Rollnick class is the set of measurable potentials $V$ whose Rollnick norm
\be
\|V\|_{\mc R} = \int_{(\set R^3)^2} \frac {|V(x)| |V(y)|}{|x-y|^2} \dd x \dd y
\ee
is finite. The Rollnick class $\mc R$ contains $L^{3/2}$. For a potential $V \in \mc R$, the Birman-Schwinger operator $\widehat T_{V_2, V_1}(\lambda)$ is Hilbert-Schmidt for every value of $\lambda$ in the lower half-plane up to the boundary.

Since $V \in L^{3/2, \infty}_0$, by Proposition \ref{prop_a3} $V_1$ and $V_2$ are in $L^{3, \infty}_0$ so there exist $V_1^n \to V_1$ and $V_2^n \to V_2$ in $L^{3, \infty}$ bounded and of compact support. It follows that $\widehat T_{V_2, V_1}$ is compact whenever $V$ is in $L^{3/2, \infty}_0$.

By the analytic and meromorphic Fredholm theorems (for statements see \cite{reesim}, p.\ 101, and \cite{reesim4}, p.\ 107), the exceptional set $\mc E$ is closed, bounded, and consists of at most a discrete set outside $\sigma(\mc H_0)$, which may accumulate toward $\sigma(\mc H_0)$, and a set of measure zero contained in $\sigma(\mc H_0)$.

Assuming that $V \in L^{3/2, \infty}_0$ is real-valued and scalar, the exceptional set resides on the real line. Indeed, if $\lambda$ is exceptional, then by the Fredholm alternative (\cite{reesim1}, p.\ 203) the equation
\be
f = -V_2 R_0(\lambda) V_1 f
\ee
must have a solution $f \in L^2$. Then $g = R_0(\lambda) V_1 f$ is in $|\dl|^{-2} L^{6/5, 2} \subset L^{6, 2}$ and satisfies
\be
g = - R_0(\lambda) V g.
\lb{3.471}
\ee

If $\lambda \in \mc E \setminus \sigma(\mc H_0)$, the kernel's exponential decay implies that $\lambda$ is an eigenvalue for $\mc H$ and that the corresponding eigenvectors must be at least in $\langle \dl \rangle^{-2} L^{6/5, 2}$; in fact, by Agmon's argument, they have exponential decay.

Furthermore, by applying $\mc H_0 -\lambda$ to both sides we obtain
\be
(\mc H_0 + V - \lambda) g = 0.
\ee
In the case of a real scalar potential $V$, $\mc H_0+V$ is self-adjoint, so this is a contradiction for $\lambda \not \in \set R$. In general, exceptional values off the real line can indeed occur.

For real-valued $V \in L^{3/2, \infty}_0$ having the matrix form (\ref{3.2}), 
any embedded exceptional values must be eigenvalues, following the argument of Lemma 4 of Erdogan--Schlag \cite{erdsch2}.

Explicitly, consider $\lambda \in \mc E \cap \sigma(\mc H_0) \setminus \{\pm \mu\}$; without loss of generality $\lambda>\mu$. It corresponds to a nonzero solution $G \in L^{6, 2}$ of
\be\lb{eq_3.75}
G = - R_0(\lambda-i0) V G.
\ee
We show that $G \in L^2$ and that it is an eigenfunction of $\mc H$. Let
\be
G = \bpm g_1 \\ g_2 \epm,\ V = \bpm W_1 & W_2 \\ -W_2 & -W_1 \epm,\ \mc H_0 = \bpm \Delta - \mu & 0 \\ 0 & -\Delta + \mu \epm,
\ee
where $W_1$ and $W_2$ are real-valued. We expand (\ref{eq_3.75}) accordingly into
\be\begin{aligned}\lb{eq_3.77}
g_1 &= (-\Delta + \lambda + \mu)^{-1}(W_1 g_1 + W_2 g_2) \\
g_2 &= (-\Delta - (\lambda - \mu - i0))^{-1}(W_2 g_1 + W_1 g_2).
\end{aligned}\ee
This implies that $g_1 \in L^1 \cap L^6$ and
\be\begin{aligned}
\langle g_2, W_2 g_1 + W_1 g_2 \rangle &= \langle (-\Delta - (\lambda - \mu - i0)^{-1}(W_2 g_1 + W_1 g_2), (W_2 g_1 + W_1 g_2) \rangle, \\
\langle g_1, W_2 g_2 \rangle &= \langle (-\Delta + \lambda + \mu)^{-1}(W_1 g_1 + W_2 g_2), W_2 g_2 \rangle, \\
\langle g_1, W_1 g_1 \rangle &= \langle (-\Delta + \lambda + \mu)^{-1}(W_1 g_1 + W_2 g_2), W_1 g_1 \rangle.
\end{aligned}\ee
However,
\be
\langle g_2, W_2 g_1 + W_1 g_2 \rangle = \ov{\langle g_1, \ov W_2 g_2 \rangle} + \langle g_2, W_1 g_2 \rangle. 
\ee
Since $W_2$ is real-valued, it follows that
\be
\langle (-\Delta - (\lambda - \mu - i0)^{-1}(W_2 g_1 + W_1 g_2), (W_2 g_1 + W_1 g_2) \rangle
\ee
is real-valued. Therefore the Fourier transform vanishes on a sphere:
\be
(W_2 g_1 + W_1 g_2)^{\wedge}(\xi) = 0
\ee
for $|\xi|^2 = \lambda-\mu$. We then apply Agmon's bootstrap argument, as follows. By Corollary 13 of \cite{golsch}, if $f \in L^1$ has a Fourier transform that vanishes on the sphere, meaning $\hat f(\xi) = 0$ for every $\xi$ such that $|\xi|^2 = \lambda \ne 0$, then
\be
\|R_0(\lambda \pm i0) f\|_2 \leq C_{\lambda} \|f\|_1.
\ee
Interpolating between this and
\be
\|R_0(\lambda \pm i0) f\|_{4} \leq C_{\lambda} \|f\|_{4/3},
\ee
which holds without conditions on $\hat f$, we obtain that for $1 < p < 4/3$ and for $\hat f = 0$ on the sphere of radius $\sqrt \lambda>0$
\be
\|R_0(\lambda \pm i0) f\|_{L^{2p/(2-p), 2}} \leq C_{\lambda} \|f\|_{L^{p, 2}}.
\ee
Thus, starting with the right-hand side of (\ref{eq_3.77}) in $L^{6/5, 2}$, we obtain that $g_2 \in L^{3, 2}$, a gain over $L^{6, 2}$. Iterating twice, we obtain that $g_2 \in L^2$. Therefore $g$ is an $L^2$ eigenvector.

Thus, for a real-valued $V \in L^{3/2, \infty}_0$ having the matrix form (\ref{3.2}), the exceptional set consists only of eigenvalues, potentially together with the endpoints of the continuous spectrum $\pm \mu$.

For a complex potential of the form (\ref{3.2}), neither of the previous arguments holds. Embedded exceptional values can occur and they need not be eigenvalues.

Next, we examine symmetries of the exceptional set $\mc E$. If $V$ is real-valued and scalar, we have already characterized $\mc E$ as being situated on the real line. If $V$ is scalar, but complex-valued, then consider an exceptional value $\lambda$, for which, due to compactness, there exists $f \in L^2$ such that
\be
f = -|V|^{1/2} \sgn V R_0(\lambda) |V|^{1/2} f.
\ee
Then
\be
(\sgn V \ov f) = -|V|^{1/2} R_0(\ov \lambda) |V|^{1/2} \sgn {\ov V} (\sgn V \ov f),
\ee
so the adjoint has an exceptional value at $\ov \lambda$. However, $\sigma(\widehat T_{V_2, V_1}(\lambda)) = \sigma(\widehat T_{V_1, V_2}(\lambda)^*)$, so all this proves that the exceptional set $\mc E$ is symmetric with respect to the real axis.

If $V$ has the matrix form (\ref{3.2}), then note that $\sigma_1 V \sigma_1 = -\ov V$, $\sigma_3 V \sigma_3 = V^*$, where $\sigma_1$ is the Pauli matrix
\be
\sigma_1 = \bpm 0 & 1 \\ 1 & 0 \epm,\ \sigma_1 \sigma_3 = -\sigma_3 \sigma_1.
\ee
Let $\lambda$ be an exceptional value, for which
\be
f = -\sigma_3 (\sigma_3 V)^{1/2} R_0(\lambda) (\sigma_3 V)^{1/2} f.
\ee
Here $\sigma_3 V = \bpm W_1 & W_2 \\ \ov W_2 & W_1 \epm$ is a selfadjoint matrix.

Then
\be\begin{aligned}
\ov f &= -\sigma_3 (\sigma_3 \ov V)^{1/2} R_0(\ov \lambda) (\sigma_3 \ov V)^{1/2} \ov f \\
&= -\sigma_3 (\sigma_3 V)^{1/2} \sigma_3 R_0(\ov \lambda) \sigma_3 (\sigma_3 V)^{1/2} \sigma_3 \ov f \\
&= -\sigma_3 (\sigma_3 V)^{1/2} R_0(\ov \lambda) (\sigma_3 V)^{1/2}\sigma_3 \ov f
\end{aligned}\ee
since $R_0$ commutes with $\sigma_3$, so whenever $\lambda$ is an exceptional value so is $\ov \lambda$.

If $V$ as in (\ref{3.2}) is a real-valued matrix, then by the same methods we obtain that $-\lambda$ is an exceptional value whenever $\lambda$ is an exceptional value.
\end{proof}

\subsection{The time evolution and projections}\lb{sect_1.35}
We begin with a basic lemma, which applies equally in the time-dependent case.
\begin{lemma}\lb{lemma_24}
Assume $V \in L^{\infty}$ and the Hamiltonian is described by (\ref{eq_3.2}) or (\ref{3.2}). Then the equation
\be
i \partial_t Z + \mc H Z = F,\ Z(0) \text{ given},
\lb{1.358}\ee
admits a weak solution $Z$ for $Z(0) \in L^2$, $F \in L^{\infty}_t L^2_x$ and
\be
\|Z(t)\|_2 \leq C e^{t \|V\|_{\infty}} \|Z(0)\|_2 + \int_0^t e^{(t-s) \|V\|_{\infty}} \|F(s)\|_2 \dd s.
\lb{1.359}\ee
\end{lemma}
\begin{proof} We introduce an auxiliary variable and write
\be
i \partial_t Z + \mc H_0 Z = F - V Z_1,\ Z(0) \text{ given}.
\ee
Over a sufficiently small time interval $[T, T+\epsilon]$, whose size $\epsilon$ only depends on $\|V\|_{\infty}$, the map that associates $Z$ to some given $Z_1$ is a contraction, in a sufficiently large ball in $L^{\infty}_t L^2_x$. The fixed point of this contraction mapping is then a solution to (\ref{1.358}).

This shows that the equation is locally solvable and, by bootstrapping, since the length of the interval is independent of the size of $F$ and of the initial data $Z(T)$, we obtain an exponentially growing global solution. The bound (\ref{1.359}) follows by Gronwall's inequality.
\end{proof}

For a nonselfadjoint operator such as $\mc H$ given by (\ref{3.2}), the projections on various parts of the spectrum need not be selfadjoint operators. The following lemma characterizes such Riesz projections (following \cite{schlag}, where it appeared in a different form).
\begin{lemma}\lb{pp} Assume $V \in L^{3/2, \infty}_0$.\begin{enumerate}\item[i)] To each element $\zeta$ of the exceptional set of $\mc H$ outside of $\sigma(\mc H_0)$ there corresponds a family of operators
\be
P^k_{\zeta} = \frac 1 {2\pi i} \int_{|z-\zeta| = \epsilon} R_V(z) (z-\zeta)^k \dd z.
\lb{3.111}
\ee
They have finite rank, $P^k_{\zeta} = 0$ for all $k \geq n$, for some $n$, $P^0_{\zeta}=(P^0_{\zeta})^2$, and more generally $(P_{\zeta}^k)(P_{\zeta}^{\ell}) = P_{\zeta}^{k+\ell}$.
\item[ii)] The ranges of $P^k_{\zeta}$ and of their adjoints $(P^k_{\zeta})^*$ are spanned by exponentially decaying functions that also belong to $\langle \dl \rangle ^{-2} L^{6/5, 2} \subset L^{6, 2}$.\\
Thus, each such projection is bounded from $L^{6/5, 2} + L^{\infty}$ to $L^1 \cap L^{6, 2}$.
\end{enumerate}
\end{lemma}
Functions in $\Ran P^k_{\zeta}$ are called \emph{generalized eigenfunctions} of $\mc H$.

We also refer the reader to Hundertmark-Lee, \cite{hule}, who proved the $L^2$ exponential decay of (generalized) eigenfunctions in the gap $-\mu<\Re \zeta<\mu$ under more general conditions.

\begin{proof} If $R_V(\zeta) = (\mc H_0 + V - \zeta)^{-1}$ exists as a bounded operator from $L^{6/5, 2}$ to $L^{6, 2}$, then $\zeta$ is not in the exceptional set and vice-versa, as a consequence of (\ref{eq_3.64}) and (\ref{eq_3.65}).

Form the contour integral, following Schlag \cite{schlag} and Reed--Simon \cite{reesim1},
\be
P^k_{\zeta} = \frac 1 {2\pi i} \int_{|z-\zeta| = \epsilon} R_V(z) (z-\zeta)^k \dd z.
\ee
This integral is independent of $\epsilon$ if $\epsilon$ is sufficiently small and $P^k_{\zeta} = 0$ for $k \geq n$. Using the Cauchy integral, it immediately follows that $(P_{\zeta}^k)(P_{\zeta}^{\ell}) = P_{\zeta}^{k+\ell}$. Furthermore,
\be\lb{3.135}
\mc H P^0_{\zeta} = P^1_{\zeta} + \zeta P^0_{\zeta}.
\ee


It is a consequence of Fredholm's theorem that the range of $P^0_{\zeta}$ is finite dimensional; from (\ref{3.111}) it follows that $\Ran(P^0_{\zeta}) \subset L^2 \cap L^{6, 2}$. Also, $\Ran(P^0_{\zeta})$ is the generalized eigenspace of $\mc H - \zeta$, meaning
\be
\Ran(P^0_{\zeta}) = \bigcup_{k \geq 0} \Ker((\mc H - \zeta)^k).
\ee
One inclusion follows from (\ref{3.135}) and the fact that $P^k_{\zeta} = 0$ for $k \geq n$. The other inclusion is a consequence of the fact that, if $(\mc H - \zeta) f = 0$, then $R_V(z) f = (\zeta - z)^{-1} f$ and, using the definition (\ref{3.111}), $P^0_{\zeta} f = f$. For higher values of $k$ we proceed by induction.

Furthermore, $\Ran(P^0_{\zeta})$ consists of functions in $\langle \dl \rangle^{-2} L^{6/5, 2}$. If $f$ is a generalized eigenfunction, meaning $f \in L^2 \cap L^{6, 2}$ and $(\mc H - \zeta)^n f = 0$, then
\be
(\mc H_0 + V) f = \zeta f + g,
\ee
where $g$ is also a generalized eigenfunction. Assuming by induction that $g \in \langle \dl \rangle^{-2} L^{6/5, 2}$ (or is zero, to begin with), the same follows for $f$.

Likewise, if $g$ is in $L^1$ or exponentially decaying, we can infer the same about $f$. Indeed, assume that $g \in e^{-\epsilon |x|} L^{6/5, 2}$, for some small $\epsilon$. Note that
\be
f = R_0(\zeta) (-V f + g) = (I + R_0(\zeta) V_2)^{-1}R_0(\zeta) (-V_1 f + g),
\ee
where $V = V_1 + V_2$. We choose $V_1$ and $V_2$ such that $V_1$ has compact support and $V_2$ is small in the $L^{3/2, \infty}$ norm. It follows that $R_0(\zeta) (-V_1 f + g)$ is in $e^{-\epsilon |x|} L^{6, 2}$, while $(I + R_0(\zeta) V_2)^{-1}$, given by an infinite Born series, is a bounded operator on the same space.

Thus, $f \in e^{-\epsilon |x|} L^{6, 2}$ is exponentially decaying. By suitably choosing a finite sequence of epsilons, the conclusion follows for all the generalized eigenfunctions associated to $\zeta$.

The range of $(P^0_{\zeta})^*$ is the generalized eigenspace of $\mc H^* - \ov \zeta$, which means that it is also finite-dimensional and spanned by exponentially decaying functions in $\langle \dl \rangle^{-2} L^{6/5, 2}$.
\end{proof}

Throughout the sequel, we assume that there are no exceptional values embedded in $\sigma(\mc H_0)$. By Fredholm's analytic theorem, this implies that there are finitely many exceptional values overall.

Then we can define $P_c$, the projection on the continuous spectrum, simply as the identity minus the sum of all projections corresponding to the exceptional values (which coincide, in this case, with the point spectrum):
\be\lb{pc}
P_c = I - P_p = I - \sum_{k=1}^n P_{\zeta_k}^0.
\ee
$P_c$ commutes with $\mc H$ and with $e^{it \mc H}$, as a direct consequence of the definition and of Lemma \ref{pp}.

In order to characterize $P_c$, we employ the subsequent lemma, which appeared in Schlag \cite{schlag} under more stringent assumptions.
\begin{lemma}\lb{lemma_31} Consider $V \in L^{3/2, \infty}_0$ and assume that there are no exceptional values of $\mc H$ embedded in the spectrum of $\mc H_0$. Then for sufficiently large $y$
\be
\langle f, g \rangle = \frac i {2\pi} \int_{\set R} \langle (R_V(\lambda + iy) - R_V(\lambda - iy)) f, g \rangle \dd \lambda
\ee
and the integral is absolutely convergent. Furthermore,
\be\lb{2.93}
\langle f, g \rangle = \frac i {2\pi} \int_{\sigma(\mc H_0)} \langle (R_V(\lambda + i0) - R_V(\lambda - i0)) f, g \rangle \dd \lambda + \sum_{k=1}^n \langle P^0_{\zeta_k} f, g \rangle
\ee
where $P^0_{\zeta_k}$ are projections corresponding to the finitely many eigenvalues $\zeta_k$.

\end{lemma}

\begin{proof}
Assume at first that $V \in L^{\infty}$ and take $y > \|V\|_{\infty}$. Then
\be
I + V_2 R_0(\lambda \pm iy) V_1
\ee
must be invertible. Indeed, $V_1$ and $V_2$ are bounded $L^2$ operators of norm at most $\|V\|^{1/2}_{\infty}$ and
\be
\|R_0(\lambda \pm iy)\|_{2 \to 2} \leq \frac 1 {4 \pi} \int_{\set R^3} \frac{e^{-\sqrt y|x|}}{|x|} \dd x = 1/y.
\ee
Therefore one can construct the inverse $(I + V_2 R_0(\lambda \pm iy) V_1)^{-1}$ by means of a power series. Thus
\be\begin{aligned}
R_V(\lambda\pm iy) &= R_0(\lambda \pm iy) - R_0(\lambda \pm iy) V R_0(\lambda \pm iy) + \\
&+ R_0(\lambda\pm iy) V_1 (I + V_2 R_0(\lambda\pm iy) V_1)^{-1} V_2 R_0(\lambda \pm iy)
\end{aligned}\ee
is a bounded $L^2$ operator.

By Lemma \ref{lemma_24} $\chi_{t \geq 0} \langle e^{it \mc H} e^{-yt} f, g \rangle$ is an exponentially decaying function of $t$ and its Fourier transform is
\be
\int_0^{\infty} \langle e^{-(y+i\lambda)t} e^{it \mc H} f, g \rangle \dd y = -i\langle R_V(\lambda-iy) f, g \rangle.
\ee
Combining this with the analogous result for the positive side, we see that
\be\lb{3.20}
(\langle e^{it \mc H} e^{-y|t|} f, g \rangle)^{\wedge} = i \langle (R_V(\lambda+iy)-R_V(\lambda-iy)) f, g \rangle.
\ee
The right-hand side is absolutely integrable, because
\be\begin{aligned}\lb{3.99}
R_V(\lambda) &= R_0(\lambda) - R_0(\lambda) V R_0(\lambda) + R_0(\lambda) V_1 (I + V_2 R_0(\lambda) V_1)^{-1} V_2 R_0(\lambda)
\end{aligned}\ee
and
\be\begin{aligned}\lb{3.100}
&\int_{-\infty}^{\infty} |\langle (R_0(\lambda+iy) - R_0(\lambda-iy)) f, g \rangle| \dd \lambda \\
&\leq \int_{-\infty}^{\infty} \frac 1 {2i} (\langle (R_0(\lambda+iy) - R_0(\lambda-iy)) f, f \rangle + \langle (R_0(\lambda+iy) - R_0(\lambda-iy)) g, g \rangle) \dd \lambda \\
&= \frac 1 2 (\|f\|_2^2 + \|g\|_2^2).
\end{aligned}\ee
The remaining terms are absolutely integrable due to smoothing estimates:
\be
\int_{-\infty}^{\infty} \||V|^{1/2} R_0(\lambda \pm iy) f\|_2^2 \dd \lambda \leq C \|f\|_2^2.
\ee

By the Fourier inversion formula, (\ref{3.20}) implies
\be\lb{3.103}
\frac i {2\pi} \int_{\set R} \langle (R_V(\lambda + iy) - R_V(\lambda - iy)) f, g \rangle \dd \lambda = \langle f, g \rangle.
\ee
We then shift the integration contour toward the essential spectrum $\sigma(\mc H_0)$, leaving behind circular contours around the finitely many (by Fredholm's Theorem) eigenvalues. Each contour integral becomes a corresponding Riesz projection.

What is left is $P_c$, the projection on the continuous spectrum. The integral is still absolutely convergent due to (\ref{3.99}), (\ref{3.100}), and smoothing estimates. (\ref{2.93}) follows.

In the beginning we assumed that $V \in L^{\infty}$. Now consider the general case $V \in L^{3/2, \infty}_0$ and a sequence of approximations by bounded potentials $V^n = V_1^n V_2 ^n \in L^{\infty}$, such that $\|V^n - V\|_{L^{3/2, \infty}} \to 0$ as $n \to \infty$. Let $\mc E$ be the exceptional set of $V$. On the set $\{\lambda \mid d(\lambda, \mc E) \geq \epsilon\}$, the norm
\be
\|(I + V_2^n R_0(\lambda) V_1^n)^{-1}\|_{L^2 \to L^2}
\ee
is uniformly bounded for large $n$. For some sufficiently high $n$, then, $\mc E(V_n) \subset \{\lambda \mid d(\lambda, \mc E) < \epsilon\}$. If
\be
y_0 = \sup \{|\Im \lambda| \mid \lambda \in \mc E\},
\ee
then for any $y > y_0$ and sufficiently large $n$
\be\lb{3.104}
\frac i {2\pi} \int_{\set R} \chi(\lambda) \langle (R_{V^n}(\lambda + iy) - R_{V^n}(\lambda - iy)) f, g \rangle \dd \lambda = \langle f, g \rangle.
\ee
Both for $V$ and for $V^n$ the integrals (\ref{3.103}) and (\ref{3.104}) converge absolutely and as $n \to \infty$ (\ref{3.104}) converges to (\ref{3.103}). To see this, subtract the corresponding versions of (\ref{3.99}) from one another and evaluate.

This proves (\ref{2.93}) for potentials $V \in L^{3/2, \infty}_0$, under the spectral assumption concerning the absence of embedded eigenvalues. 
\end{proof}

By Lemma \ref{lemma_31}, it follows that
\be\lb{2.94}
P_c = \chi(\mc H) = \frac i {2\pi} \int_{\sigma(\mc H_0)} \big(R_V(\lambda+i0) - R_V(\lambda-i0)\big) \dd \lambda.
\ee
$P_c$ is bounded on $L^2$, but, since each projection $P_{\zeta_k}^0$ is bounded from $L^{\infty} + L^{6/5, 2}$ to $L^{6, 2} \cap L^1$, the same holds for $P_p = I-P_c$.

Therefore $P_c$ is bounded on $L^{6/5, q}$, $q \leq 2$, and on $L^{6, q}$, $q\geq 2$, as well as on intermediate spaces.





\subsection{Technical lemmas} In order to apply Wiener's Theorem --- Theorem \ref{thm_1.6} ---, we first need to exhibit an element of $K$. By Proposition \ref{prop_21},
\be
T_{V_2, V_1}(t) = V_2 e^{it \mc H_0} V_1
\ee
is a such an element. However, another condition is that the Fourier transform $I - i \widehat T_{V_2, V_1}$ should be invertible at every $\lambda \in \set R$. By Lemma \ref{lemma_22},
\be
I - i \widehat T_{V_2, V_1}(\lambda) = I + V_2 R_0(\lambda) V_1
\ee
and this is invertible for each $\lambda$, except for $\lambda \in \mc E$. This suffices only if there are no exceptional values; otherwise we need a correction.

By (\ref{eq_3.64}), we see that the problem is that $R_V$ is not uniformly bounded in the lower half-plane in $\mc L(L^{6/5, 2}, L^{6, 2})$. The solution will be to replace $R_V$ with $R_V P_c - (\lambda + i \delta)^{-1} P_p$, which is uniformly bounded in the lower half-plane when $\delta>0$.

Our construction involves using, instead of $V_1$ and $V_2$, the following modified versions:


\begin{lemma}\lb{lemma32}
Consider $V \in L^{3/2, \infty}_0(\set R^3)$ and $\mc H = \mc H_0 + V$ as in (\ref{eq_3.2}) or (\ref{3.2}) such that $\mc H$ has no exceptional values embedded in $\sigma(\mc H_0)$. Then, for any $\delta \in \set C$, there exists a decomposition
\be
V - P_p (\mc H - i\delta) = \tilde V_1 \tilde V_2,
\ee
where $P_p = I - P_c$, such that $\tilde V_1$, $\tilde V_2^* \in \mc L(L^2, L^{6/5, 2})$, and the Fourier transform of $I-iT_{\tilde V_2, \tilde V_1}$ is uniformly invertible in the lower half-plane:
\be\lb{2.108}
\sup_{\Im \lambda \leq 0} \|I - i \widehat T_{\tilde V_2, \tilde V_1}(\lambda)\|_{\mc L(L^2, L^2)} < \infty.
\ee
Here
\be
(T_{\tilde V_2, \tilde V_1}F)(t) = \int_{-\infty}^t \tilde V_2 e^{i(t-s)\mc H_0} \tilde V_1 F(s) \dd s.
\ee
Furthermore, $\tilde V_1$ and $\tilde V_2^*$ can be approximated in the $\mc L(L^2, L^{6/5, 2})$ norm by operators that are bounded from $L^2$ to $\langle x \rangle^{-N} L^2$, for any fixed $N$.

Finally, if $V \in L^{3/2, 1}$ then $L^{6/5, 2}$ can be replaced by $L^{6/5, 1}$.
\end{lemma}
Note that $\tilde V_1$ and $\tilde V_2$ are the same as $V_1$, respectively $V_2$, up to exponentially small perturbations; however, while $V_1$ and $V_2$ are functions, $\tilde V_1$ and $\tilde V_2$ are operators.

Our interest in them is explained by (\ref{2.108}), which is not guaranteed for  $V_1$ and $V_2$ by themselves, as explained above.


\begin{proof}[Proof of Lemma \ref{lemma32}]
Consider a potential $V$ such that $\mc H$ has no exceptional values embedded in $\sigma(\mc H_0)$. Having finite rank, $P_p$ has the form
\be
P_p f = \sum_{k=1}^n \langle f, f_k \rangle g_k,
\ee
where $f_k$ and $g_k$ belong to $\langle \dl \rangle^{-2} L^{6/5, 2}$. 
Take the standard polar decomposition of $P_p$, with respect to $L^2$:
\be
P_p = U A,
\ee
where $A = (P_p^* P_p)^{1/2} \geq 0$ is a nonnegative $L^2$ operator of finite rank. More specifically, $A$ has the form
\be
A f = \sum_{j, k=1}^n a_{jk} \langle f, f_k \rangle f_j.
\ee
$A$ maps to the span of $f_k$, while $U$ is a partial $L^2$ isometry defined on the range of $A$. $U$ maps the span of $f_k$ to the span of $g_k$ and can be extended by zero on the orthogonal complement:
\be\begin{aligned}
U f &= \sum_{j, k=1}^n u_{jk} \langle f, f_k \rangle g_j.
\end{aligned}\ee
From these explicit forms we see that both $U$ and $A$ are bounded operators from $\langle \dl \rangle^2 L^{6, 2}$, which includes $L^2$, to $\langle \dl \rangle^{-2} L^{6/5, 2}$.


Also let $V = V_1 V_2$, where $V_2 \geq 0$ is a nonnegative operator on $L^2$, meaning $\langle f, V_2 f \rangle \geq 0$ for every $f \in \Dom(V_2)$, and $V_1$, $V_2 \in L^{3, \infty}$.

Then, define the bounded $L^2$ operators $G_1 = V_2/(V_2 + A)$ and $G_2 = (\mc H + i\delta) P_p/(V_2 + A)$, initially on $\Ran(V_2 + A)$, by setting
\be
G_1((V_2+A)f) = V_2 f,\ G_2((V_2+A)f) = (\mc H +\delta) P_p f
\ee
and extending them by continuity to $\ov{\Ran(V_2 + A)}$. On the orthogonal complement
\be
\Ran(V_2 + A)^{\perp}=\{f \mid P_p f =0,\ V f = 0\}
\ee
we simply set $G_1=G_2=0$. We then make the construction
\be\begin{aligned}
V - P_p (\mc H-i\delta) &= \tilde V_1 \tilde V_2, \\
\tilde V_2 &= V_2 + A, \\
\tilde V_1 &= V_1 G_1 - G_2.
\end{aligned}\ee
Upon inspection, $\tilde V_1$ and $\tilde V_2^*$ are bounded from $L^2$ to $L^{6/5, 2}$.

We next prove that $\tilde V_1$ and $\tilde V_2$ can be approximated by operators in better classes, as claimed. Firstly, consider a family of smooth compactly supported functions $\chi_n$ such that $0 \leq \chi_n \leq 1$ and $\chi_n \to 1$ uniformly on compact sets as $n \to \infty$. Let
\be\begin{aligned}\lb{eq_3.177}
F_1^n &= \chi_n V_1 G_1 - \chi_n G_2 \chi_n,\\
F_2^n &= \chi_n V_2 + \chi_n A \chi_n.
\end{aligned}\ee

It is plain that $F_1^n$ and $(F_2^n)^*$ take $L^2$ to $\langle x \rangle^{-N} L^2$. They also approximate $\tilde V_1$ and $\tilde V_2^*$ in $\mc L(L^2, L^{6/5, 2})$. Indeed, as $n \to \infty$
\be
\|\chi_n V_j - V_j\|_{\mc L(L^2, L^{6/5, 2})} \to 0
\ee
because $V$ decays at infinity and
\be\begin{aligned}
\|\chi_n G_2 \chi_n - G_2\|_{\mc L(L^2, L^2)} \to 0,\\
\|\chi_n A \chi_n - A\|_{\mc L(L^2, L^{6/5, 2})} \to 0,
\end{aligned}\ee
because $G_2$ and $A$ have finite rank.





Finally, we show that the Fourier transform $I-i\widehat{T}_{\tilde V_2, \tilde V_1}(\lambda)$ is always invertible. Lemma \ref{lemma_22} implies that
\be
I-i\widehat{T}_{\tilde V_2, \tilde V_1}(\lambda) = I + \tilde V_2 R_0(\lambda) \tilde V_1.
\ee
As we see by (\ref{2.97}),
\be\lb{2.106}
(I + \tilde V_2 R_0 \tilde V_1)^{-1} = I - \tilde V_2 (R_V P_c - (\lambda - i\delta)^{-1} P_p) \tilde V_1.
\ee
By (\ref{eq_3.64}) and (\ref{eq_3.65}), $R_V(\lambda)$ is bounded from $L^{6/5, 2}$ to $L^{6, 2}$ if and only if $\lambda$ is not an exceptional value. Our assumption regarding the absence of embedded exceptional values implies that $R_V(\lambda)$ is uniformly bounded for $\lambda \in \sigma(\mc H_0)$. Furthermore, $R_V$ is uniformly bounded away from the finitely many exceptional values.

Using the representation formula (\ref{2.94}) for $P_c$, for $f$, $g \in L^2$
\be\begin{aligned}
\langle R_V(\lambda_0) P_c f, g \rangle &= \frac 1 {2 \pi i} \int_{\sigma(\mc H_0)} \big\langle R_V(\lambda_0) (R_V(\lambda-i0) - R_V(\lambda+i0)) f, g \big\rangle \dd \lambda \\
&= \frac 1 {2 \pi i} \int_{\sigma(\mc H_0)} \Big\langle \frac 1 {\lambda-\lambda_0} (R_V(\lambda-i0) - R_V(\lambda+i0)) f, g \Big\rangle \dd \lambda
\end{aligned}\ee
and the integral converges absolutely. Here we used the resolvent identity: for all $\lambda_1$, $\lambda_2$ in the resolvent set,
$$
R_V(\lambda_1) - R_V(\lambda_2) = (\lambda_1-\lambda_2) R_V(\lambda_1) R_V(\lambda_2).
$$
For some fixed $\lambda_1 \not \in \sigma(\mc H_0)$, $R_V(\lambda_1)$ is bounded from $L^{6/5, 2}$ to $L^{6, 2}$. Then, for any other value $\lambda_2 \not \in \sigma(\mc H_0)$, one has that
\be\begin{aligned}
&\langle R_V(\lambda_1) P_c f, g \rangle - \langle R_V(\lambda_2) P_c f, g \rangle = \\
&= \frac 1 {2 \pi i} \int_{\sigma(\mc H_0)} \Big\langle \frac {\lambda_2 - \lambda_1} {(\lambda-\lambda_1)(\lambda-\lambda_2)} (R_V(\lambda-i0) - R_V(\lambda+i0)) f, g \Big\rangle \dd \lambda.
\end{aligned}\ee
Since the integrand decays like $\lambda^{-2}$, it follows that
\be\lb{lim_ap}
\sup_{\lambda \in \set C} \|R_V(\lambda) P_c\|_{L^{6/5, 2} \to L^{6, 2}} < \infty.
\ee
%

By (\ref{2.106}), $(I + \tilde V_2 R_0(\lambda) \tilde V_1)^{-1}$ is then uniformly bounded in the lower half-plane, for $\delta>0$, and (\ref{2.108}) follows.
\end{proof}

As we replace $V_1$ and $V_2$ by $\tilde V_1$, respectively $\tilde V_2$, we use the following identities:
\begin{lemma} For $\delta>0$, consider the decomposition of Lemma \ref{lemma32}:
\be
V - P_p (\mc H + i\delta) = \tilde V_1 \tilde V_2.
\ee
Let $\tilde R_V(\lambda) = (P_c \mc H - i\delta P_p - \lambda)^{-1}$.
Then
\be\lb{2.127}
\tilde R_V(\lambda) = R_V(\lambda) P_c - (\lambda - i \delta)^{-1} P_p
\ee
and the following identities hold:
\begin{align}
\lb{2.97}
\tilde R_V(\lambda) &= R_0(\lambda) - R_0(\lambda) \tilde V_1 (I + \tilde V_2 R_0(\lambda) \tilde V_1)^{-1} \tilde V_2 R_0(\lambda), \\
\lb{3.65'}
(I + \tilde V_2 R_0(\lambda) \tilde V_1)^{-1} &= I - \tilde V_2 \tilde R_V(\lambda) \tilde V_1.
\end{align}
\end{lemma}
\begin{proof} Direct computation shows that
\be\lb{2.109}
(P_c \mc H + i\delta P_p - \lambda) (R_V(\lambda) P_c - (\lambda - i \delta)^{-1} P_p) = I.
\ee
We use the fact that $P_c^2 = P_c$, $P_p^2 = P_p$, $P_c P_p = P_p P_c = 0$, and everything commutes in (\ref{2.109}).

Then, note that
\be
P_c \mc H - i\delta P_p = \mc H_0 + V - P_p (\mc H + i\delta) = \mc H_0 + \tilde V_1 \tilde V_2.
\ee
We write the resolvent identity, for this case, as
\be
\tilde R_V(\lambda) = R_0(\lambda) - R_0(\lambda) \tilde V_1 \tilde V_2 \tilde R_V(\lambda) = R_0(\lambda) - \tilde R_V(\lambda) \tilde V_1 \tilde V_2 R_0(\lambda),
\ee
which implies (\ref{2.97}) and (\ref{3.65'}).
\end{proof}

We prefer to study the modified evolution $e^{it \mc H P_c - \delta t P_p}$, rather than $e^{it \mc H}$ or $e^{it\mc H} P_c$. We define this modified evolution simply by means of the equation
\be
i \partial_t Z + \mc H P_c Z + i \delta P_p Z = 0,\ Z(0)=Z_0,
\ee
and take $e^{it \mc H P_c - \delta t P_p} Z_0 := Z(t)$.

\begin{lemma}\lb{lem214} Assume $\|V\|_{\infty} < \infty$; then
\be
e^{it \mc H P_c - \delta t P_p} = e^{it \mc H} P_c + e^{-\delta t} P_p.
\ee
For $\Im \lambda < \|V\|_{\infty}$,
\be\lb{2.135}
(e^{it \mc H P_c - \delta t P_p})^{\wedge}(\lambda) = \tilde R_V(\lambda).
\ee
\end{lemma}
\begin{proof}
By Lemma \ref{lemma_24}, this modified evolution is $L^2$-bounded when $V \in L^{\infty}$:
\be\lb{2.136}
\|e^{it \mc H P_c - \delta t P_p}\|_{\mc L(L^2, L^2)} \leq C e^{t \|V\|_{L^{\infty}_t}}.
\ee
Since $\mc H$, $P_c$, and $P_p$ all commute, we can also rewrite it as
\be
e^{it \mc H P_c - \delta t P_p} = e^{it \mc H P_c - \delta t P_p} P_c + e^{it \mc H P_c - \delta t P_p} P_p = e^{it \mc H} P_c + e^{-\delta t} P_p.
\ee
(\ref{2.136}) allows one to define the Fourier transform for $\Im \lambda < - \|V\|_{\infty}$. A direct computation then shows (\ref{2.135}).
\end{proof}

\subsection{Time-independent potentials}


\begin{proof}[Proof of Theorem \ref{theorem_26}]
Replace the original equation (\ref{1.1}) by the following:
\be\lb{2.152}
i \partial_t Z + \mc H P_c Z + i\delta P_p Z = F.
\ee
This equation is fulfilled by $P_c Z$, $P_c F$, and $P_c Z(0)$ (as $Z$, $F$, and $Z(0)$ respectively).

By Duhamel's formula, since
$$
\mc H P_c + i \delta P_p = \mc H_0 + V - P_p (\mc H - i \delta) = \mc H_0 + \tilde V_1 \tilde V_2,
$$
a solution of (\ref{2.152}) satisfies
\be\lb{2.153}
Z(t) = e^{it\mc H_0} Z(0) - i \int_0^t e^{i(t-s) \mc H_0} F(s) \dd s + i \int_0^t e^{i(t-s)\mc H_0} (V - P_p (\mc H - i \delta)) Z(s) \dd s.
\ee

Assume $V \in L^{\infty}$. Let $F$, $G \in L^{\infty}_t (L^1_x \cap L^2_x)$ have compact support in $t$ and consider the forward time evolution
\be
(T_V F)(t) = \int_{t>s} e^{i(t-s) \mc H P_c - \delta (t-s) P_p} F(s) \dd s. 
\ee
By Lemma \ref{lem214},
\be
\widehat {T_V F}(\lambda) = i \tilde R_V(\lambda) \widehat F(\lambda).
\ee
For $y > \|V\|_{\infty}$, both $e^{-yt} (T_V F)(t)$ and $e^{yt} G(t)$ belong to $L^2_{t, x}$. Taking the Fourier transform in $t$, by Plancherel's theorem
\be\begin{aligned}\lb{2.114}
\int_{\set R} \langle (T_V F)(t),  G(t) \rangle \dd t &= \frac 1 {2\pi} \int_{\set R} \big\langle \big(e^{-yt} (T_V F)(t) \big)^{\wedge}, \big(e^{yt} G(t)\big)^{\wedge} \big\rangle \dd \lambda \\
&= \frac 1 {2 \pi i} \int_{\set R} \big\langle \tilde R_V(\lambda - iy) \widehat F(\lambda-iy), \widehat {G(-t)}(\lambda-iy) \big \rangle \dd \lambda.
\end{aligned}\ee
Here $\langle \cdot, \cdot \rangle$ is the real dot product.

Following (\ref{2.97}) and (\ref{3.65'}), we express $\tilde R_V$ in (\ref{2.114}) as
\be\lb{3.42}\begin{aligned}
\tilde R_V &= R_0 - R_0 \tilde V_1 \tilde V_2 R_0 + R_0 \tilde V_1 \big(\tilde V_2 \tilde R_V \tilde V_1\big) \tilde V_2 R_0.
\end{aligned}\ee
The first term represents the free Schr\"{o}dinger evolution:
\be
\frac 1 {2 \pi i} \int_{\set R} \big\langle R_0(\lambda - iy) \widehat F(\lambda-iy), \widehat {G(-t)}(\lambda-iy) \big \rangle \dd \lambda.
\ee
It is bounded by virtue of the endpoint Strichartz estimates of \cite{tao}.

Likewise, by means of Strichartz or smoothing estimates for the free Schr\"{o}dinger equation, one obtains that
\be\begin{aligned}\lb{3.44}
\big\|\tilde V_2 R_0(\lambda-i0) \widehat F(\lambda-i0)\big\|_{L^2_{\lambda, x}} \leq C \|F\|_{L^2_t L^{6/5, 2}_x},\\
\big\|\tilde V_1^* R_0(\lambda+i0) \widehat G(\lambda-i0)\big\|_{L^2_{\lambda, x}} \leq C \|G\|_{L^2_t L^{6/5, 2}_x}.
\end{aligned}\ee
The same holds for the integral along any other horizontal line, meaning that for every $y \in \set R$, $\tilde V_2 R_0(\lambda+iy) \widehat F(\lambda+iy)$ and $\tilde V_1^* R_0(\lambda + iy) \widehat G(\lambda+iy)$ are in $L^2_{\lambda, x}$. This allows shifting the integration line toward the real axis.

Taking into account the fact that, by (\ref{2.127}) and (\ref{lim_ap}),
\be
\sup_{\Im \lambda \leq 0} \|\tilde V_2 \tilde R_V \tilde V_1\|_{\mc L(L^2_{t, x}, L^2_{t, x})} < \infty
\ee
and following (\ref{3.42}) and (\ref{3.44}), we obtain
\be\begin{aligned}
\int_{\set R} \langle (T F)(t),  G(t) \rangle \dd t &\leq C \|F\|_{L^2_t L^{6/5, 2}_x} \|G\|_{L^2_t L^{6/5, 2}_x}.
\end{aligned}\ee
Then, we remove our previous assumption that $F$ and $G$ should have compact support. This establishes the estimate
\be
\bigg\|\int_{t>s} e^{i(t-s) \mc H P_c - \delta(t-s) P_p} F(s) \dd s \bigg\|_{L^2_t L^{6, 2}_x} \leq C \|F\|_{L^2_t L^{6/5, 2}_x}.
\ee
By taking the $P_c$ projection, we obtain the more usual Strichartz retarded estimate. Then, using Duhamel's formula (\ref{2.153}), we obtain (\ref{1.45}).

Given $V \in L^{3/2, \infty}_0$, we approximate it in the $L^{3/2, \infty}$ norm by $L^{\infty}$ potentials whose exceptional sets are still disjoint from $\sigma(\mc H_0)$. Since the conclusion stands for each approximation, uniformly, we pass to the limit and it also holds for $V$ itself.

Next, we prove (\ref{1.46}). Assume $V \in L^{3/2, 1}$. With $\tilde V_1$ and $\tilde V_2$ as in Lemma \ref{lemma32} (meaning $\tilde V_1$ and $\tilde V_2^*$ are bounded from $L^2$ to $L^{6/5, 1}$), let
\be
(T_{\tilde V_2, \tilde V_1} F)(t) = \int_{-\infty}^t \tilde V_2 e^{i(t-s)\mc H_0} \tilde V_1 F(s) \dd s.
\ee
We need to establish that $T_{\tilde V_1, \tilde V_2} \in K$, where $K$ is the Wiener algebra of Definition \ref{def_k} for $H = L^2_x$. By Proposition \ref{prop_21} and Minkowski's inequality,
\be
T_{I, I} F(t) = \int_{-\infty}^t e^{i(t-s) \mc H_0} F(s) \dd s
\ee
takes $L^1_t L^{6, 1}_x$ to $L^1_t L^{6/5, \infty}_x$.
Therefore $T_{\tilde V_2, \tilde V_1}$
takes $L^1_t L^2_x$ to itself and, for the Hilbert space $H=L^2_x$, belongs to the algebra $K$. Following Lemma \ref{lemma_22}, the Fourier transform of $T_{\tilde V_2, \tilde V_1}$ is given by
\be
\widehat T_{\tilde V_2, \tilde V_1}(\lambda) = i\tilde V_2 R_0(\lambda) \tilde V_1.
\ee
By Lemma \ref{lemma32}, $I - i \widehat T_{\tilde V_2, \tilde V_1}(\lambda)$ is invertible for $\Im \lambda \leq 0$.

Considered as an operator from $L^p$, $p<6/5$, to its dual, $T_{I, I}(t)$ decays at infinity faster than $|t|^{-1}$ in norm:
\be
\|T_{I, I}(t)\|_{\mc L(L^p, L^{p'})} \leq C |t|^{3/2-3/p}
\ee
and $3/2-3/p < -1$ for $p<6/5$. Thus, if $\tilde V_1$ and $\tilde V_2^*$ were in $\mc L(L^2, L^p)$,
\be
\|\chi_{|t|>R} T_{\tilde V_2, \tilde V_1}\|_K \leq \int_t^{\infty} \|T_{\tilde V_2, \tilde V_1}(t)\|_{\mc L(L^2, L^2)} \leq C |t|^{5/2-3/p} \to 0
\ee
as $t \to \infty$. By approximating $\tilde V_1$ and $\tilde V_2$ in the operator norm, it follows that $T_{\tilde V_2, \tilde V_1}$ belongs to the subalgebra $K_0 \subset K$ of kernels that decay at infinity.

Moreover, consider approximations of $\tilde V_1$ and $\tilde V_2$, call them $F_1^n$ and $F_2^n$, as given by Lemma \ref{lemma32}. Then for each $t$
\be
\|(F_2^n e^{-i(t+\epsilon)\Delta} F_1^n - F_2^n e^{-it\Delta} F_1^n)f\|_2 \leq C \epsilon^{1/4} \|e^{-it\Delta} F_1^n f\|_{\dot H^{1/2}_{loc}}
\ee
and therefore
\be\begin{aligned}
&\int_{-t}^t \|(F_2^n e^{i(t+\epsilon)\mc H_0} F_1^n - F_2^n e^{it\mc H_0} F_1^n)f\|_2 \dd t \\
&\leq C \epsilon^{1/4} t^{1/2} \Big(\int_{-t}^t \|e^{-it\Delta} F_1^n f\|_{\dot H^{1/2}_{loc}}^2 \dd t\Big)^{1/2}\\
&\leq C \epsilon^{1/4} t^{1/2} \|f\|_2.
\lb{1.352}
\end{aligned}\ee
Since $T_{F_2^n, F_1^n}$ decays at infinity, (\ref{1.352}) implies that $T_{F_2^n, F_1^n}$ is equicontinuous. By passing to the limit, we find that the same holds for $T_{\tilde V_2, \tilde V_1}$.

Therefore $I-iT_{\tilde V_2, \tilde V_1}$ satisfies all the hypotheses of Theorem \ref{thm7}, with respect to the Hilbert space $L^2$ and the algebra $K$ of Definition \ref{def_k}. 
The inverse $(I-iT_{\tilde V_2, \tilde V_1})^{-1}$ then belongs to $K$ and is supported on $(-\infty, 0]$ in $t$ by Lemma ~\ref{lemma8}.

Let $T_{I, \tilde V_1}$ and $T_{\tilde V_2, I}$ be respectively given by
\be
(T_{I, \tilde V_1}F)(t) = \int_{-\infty}^t e^{i(t-s)\mc H_0} \tilde V_1 F(s) \dd s
\ee
and
\be
(T_{\tilde V_2, I} F)(t) = \int_{-\infty}^t \tilde V_2 e^{i(t-s)\mc H_0} F(s) \dd s.
\ee
Given the decomposition $V - P_p (\mc H + i \delta) = \tilde V_1 \tilde V_2$, (\ref{2.153}) becomes
\be\begin{aligned}
\tilde V_2 Z(t) &= \tilde V_2 \int_0^t e^{i(t-s) \mc H_0} (-iF(s) + \delta_{s=0} Z(0)) \dd s + i \int_0^t (\tilde V_2 e^{i(t-s)\mc H_0} \tilde V_1) \tilde V_2 Z(s) \dd s.
\end{aligned}\ee
Then
\be\lb{3.143}
\big(I - i T_{\tilde V_2, \tilde V_1}\big) \tilde V_2 Z = T_{\tilde V_2, 1} (-iF + \delta_{t=0} Z(0)).
\ee
Consequently, we can rewrite (\ref{2.153}) as
\be\begin{aligned}
Z &= T_{I, I} (-iF + \delta_{t=0} Z(0)) + T_{I, \tilde V_1} (I-iT_{\tilde V_2,  \tilde V_1})^{-1} T_{\tilde V_2, I} (-iF + \delta_{t=0} Z(0)).
\lb{3.51}
\end{aligned}\ee

By Proposition \ref{prop_21},
\be
\|T_{\tilde V_2, 1}(-iF + \delta_{t=0} Z(0))\|_{L^1_t L^2_x} \leq C (\|F\|_{L^1_t L^{6/5, 1}_x} + \|Z(0)\|_{L^{6/5, 1}_x})
\ee
The convolution kernel $(I-iT_{\tilde V_2, \tilde V_1})^{-1}$ then takes $L^1_t L^2_x$ into itself again, while the last operator $T_{1, \tilde V_1}$ takes $L^1_t L^{6/5, 1}_x$ into $L^1_t L^{6, \infty}_x$. We obtain that
$P_c Z \in L^1_t L^{6, \infty}_x$, as desired.
\end{proof}

For completeness, we also sketch the proof of Proposition \ref{cor1.5}.
\begin{proof}[Proof of Proposition \ref{cor1.5}]
If $V \in L^{3/2, 1}$, then (\ref{1.12}) follows directly by interpolation between (\ref{1.45}) and (\ref{1.46}), which are both satisfied. If $V \in L^{3/2}$, then we use Theorem \ref{thm6/5}.

We need to show in this case that $T_{V_2, V_1}$ belongs to $K_{6/5}$. By complex interpolation between
$$
\|T_{V_2,V_1} F\|_{L^1_t L^2_x} \leq C \|V_1\|_{L^{3, 2}} \|V_2\|_{L^{3, 2}} \|F\|_{L^1_t L^2_x},
$$
which follows by Proposition \ref{prop_21}, and
$$
\|T_{V_2,V_1} F\|_{L^2_{t, x}} \leq C \|V_1\|_{L^{3, \infty}} \|V_2\|_{L^{3, \infty}} \|F\|_{L^2_{t, x}},
$$
which follows by the endpoint Strichartz estimate for $e^{-it\Delta}$, we obtain that
$$
\|T_{V_2,V_1} F\|_{L^{6/5}_t L^2_x} \leq C \|V_1\|_{L^{3}} \|V_2\|_{L^{3}} \|F\|_{L^{6/5}_t L^2_x},
$$
so, by definition, $T_{V_2,V_1} \in K_{6/5}$.

Since $T_{\tilde V_2, \tilde V_1} - T_{V_2, V_1}$ has finite rank, a similar interpolation argument shows that $T_{\tilde V_2, \tilde V_1} \in K_{6/5}$.

From this point, the proof follows that of (\ref{1.46}).
\end{proof}


\subsection{The time-dependent case}\lb{sec3.7}
\begin{proof}[Proof of Corollary \ref{cor_1.5}]
Introduce an auxiliary function $Z_1 = P_c Z_1$ (supported on the continuous spectrum) and write (\ref{1.12}) in the form
\be\lb{1.12'}
i \partial_t Z + \mc H Z + \tilde V(t, x) P_c Z_1= F,\ Z(0) \text{ given}.
\ee
By the Strichartz estimate (\ref{1.45}),
\be\begin{aligned}\lb{2.160}
\|P_c Z\|_{L^{\infty}_t L^2_x \cap L^2_t L^{6, 2}_x} &\leq C \Big(\|Z(0)\|_2 + \|F\|_{L^1_t L^2_x + L^2_t L^{6/5, 2}_x} + \\
&+\|\tilde V\|_{L^{\infty}_t L^{3/2, \infty}_x} \|P_c Z_1\|_{L^2_t L^{6, 2}_x} \Big).
\end{aligned}\ee
For two different values of $Z_1$, call them $Z_1^1$ and $Z_1^2$, we subtract the corresponding copies of (\ref{1.12'}) from one another. The contributions of $Z(0)$ and $F$ cancel, so we obtain
\be\begin{aligned}
\|P_c Z^1 - P_c Z^2\|_{L^{\infty}_t L^2_x \cap L^2_t L^{6, 2}_x} \leq C \|\tilde V\|_{L^{\infty}_t L^{3/2, \infty}_x} \|P_c Z_1^1 - P_c Z_1^2\|_{L^2_t L^{6, 2}_x}.
\end{aligned}\ee
Thus, if $\|\tilde V\|_{L^{\infty}_t L^{3/2, 2}_x}$ is sufficiently small, the map that associates the solution $P_c Z$ to the auxiliary function $P_c Z_1$ is a contraction inside some ball of large radius in $L^2_t L^{6, 2}_x$. The fixed point of this contraction is a solution to (\ref{1.12}) and (\ref{2.160}) then implies (\ref{1.13}).

The proof of (\ref{1.14}) is entirely analogous.
\end{proof}


Consider the family of isometries $U$ given by (\ref{1.20}) or (\ref{1.21}).
In the subsequent lemma we encapsulate the properties of $U$ that we actually use in our study of (\ref{3.157}) and (\ref{3.158}).

\begin{lemma}\lb{lem_32} Let $U(t)$ be defined by (\ref{1.20}) or (\ref{1.21}). Then $U$ possesses properties P1--P4 listed in Theorem \ref{theorem_13}.
\end{lemma}
\begin{proof} With no prejudice, we only consider the matrix case (\ref{1.20}).

It is straightforward to verify properties P1 and P2: $U(t)$ are isometries, so they preserve any isometry-invariant Banach space. The Laplace operator $\Delta$ (hence $\mc H_0$, too) commutes with translations and rotations.


The third property, P3, concerns a comparison between
\be\begin{aligned}
T(t, s) = e^{i(t-s) \mc H_0}
\end{aligned}\ee
and
\be\begin{aligned}
\tilde T(t, s) &= e^{i(t-s) \mc H_0} U(t) U(s)^{-1} \\
&= e^{i(t-s) \mc H_0} \Omega(t) e^{(D(t) - D(s)) \dl + i (A(t) - A(s)) \sigma_3} \Omega(s)^{-1}.
\end{aligned}\ee


We first study the model operator with $\Omega \equiv I$ and $D \equiv 0$, so we only have to handle the oscillation $A$. Denote this by
\be\begin{aligned}
\tilde T_1(t, s) &= e^{i(t-s) \mc H_0} e^{i(A(t) - A(s)) \sigma_3}.
\end{aligned}\ee
One has
\be
e^{ia} -1 \leq C \min(1, a)
\ee
and then, by the $L^1$ to $L^{\infty}$ $t^{-3/2}$ dispersive estimate, for large enough $N$
\be\lb{3.167}\begin{aligned}
\|\langle x \rangle^{-N} (T(t, s) - \tilde T_1(t, s)) \langle x \rangle^{-N}\|_{2 \to 2} &\leq C \|T(t, s) - \tilde T(t, s)\|_{1 \to \infty} \\
&\leq C |t-s|^{-3/2} |A(t)-A(s)| \\
&\leq C |t-s|^{-1/2} \|A'\|_{\infty}.
\end{aligned}\ee

Next, consider the case when $\Omega \equiv I$, but $D$ and $A$ need not be zero. Let
\be
\tilde T_2(t, s) = e^{i(t-s) \mc H_0} e^{(D(t) - D(s)) \dl + i (A(t) - A(s)) \sigma_3}.
\ee
We start with the Leibniz-Newton formula
\be
\tilde T_2(t, s) - \tilde T_1(t, s) = \int_s^t D'(\tau) \dl e^{i(t-s) \mc H_0} e^{(D(\tau) - D(s)) \dl + i (A(t) - A(s)) \sigma_3} \dd \tau,
\ee
which implies
\be\begin{aligned}\lb{eq2.166}
&\|\langle x \rangle^{-N} (\tilde T_2(t, s) - \tilde T_1(t, s)) \langle x \rangle^{-N}\|_{2 \to 2} \leq \\
&\leq C \|D'\|_{\infty} |t-s| \cdot \sup_{\tau \in [s, t]} \|\langle x \rangle^{-N} \dl e^{(D(\tau) - D(s)) \dl} \tilde T_1(t, s) \langle x \rangle^{-N}\|_{2 \to 2} \\
&= C \|D'\|_{\infty} |t-s| \cdot \sup_{\tau \in [s, t]} \|\langle x \rangle^{-N} \dl e^{(D(\tau) - D(s)) \dl} \dl T(t, s) \langle x \rangle^{-N}\|_{2 \to 2}.
\end{aligned}\ee
From the explicit form of the fundamental solution of the free Schr\"{o}dinger equation,
we obtain the following bound for the kernel of $\dl T(t, s)$:
\be
|\dl T(t, s)|(x, y) \leq C \frac {|x-y|}{|t-s|^{5/2}}.
\ee
Then,
\be\begin{aligned}\lb{2.168}
|\dl e^{(D(\tau)-D(s)) \dl} T(t, s)|(x, y) &= |\dl T(t, s)|(x+D(\tau)-D(s), y) \\
&\leq C \frac {|x| + |y| + |\tau-s| \|D'\|_{\infty}}{|t-s|^{5/2}}.
\end{aligned}\ee
(\ref{eq2.166}) and (\ref{2.168}) imply that, for sufficiently large $N$,
\be\lb{2.169}
\|\langle x \rangle^{-N} (\tilde T_2(t, s) - \tilde T_1(t, s)) \langle x \rangle^{-N}\|_{2 \to 2} \leq C (|t-s|^{-3/2} \|D'\|_{\infty} + |t-s|^{-1/2} \|D'\|_{\infty}^2).
\ee
Finally, if $\Omega \ne I$, let $\tilde D(\tau) = \Omega(s) D(\tau)$; then
\be
\tilde T(t, s) = e^{i(t-s) \mc H_0} \Omega(t) \Omega(s)^{-1} e^{(\tilde D(t) - \tilde D(s))\dl + i(A(t) - A(s))\sigma_3}.
\ee
We use a concrete description of $\Omega'(t)$: there exists an infinitesimal rotation
\be
\omega(t) = (x \cdot e_1(t)) \partial_{e_2(t)} - (x \cdot e_2(t)) \partial_{e_1(t)}
\ee
such that $\partial_t \Omega(t)$ equals $|\Omega'(t)|$ times $\omega(t)$, applied to $\Omega(t)$:
\be\lb{omega}
\partial_t \Omega(t) f = |\Omega'(t)| \omega(t) \Omega(t) f.
\ee
By the Leibniz-Newton formula again,
\be\begin{aligned}
\tilde T(t, s) - \tilde T_2(t, s) &= \int_s^t |\Omega'(\tau)| \omega(t) e^{i(t-s) \mc H_0} \Omega(\tau) \Omega(s)^{-1} \\
&e^{(\tilde D(t) - \tilde D(s))\dl + i(A(t) - A(s))\sigma_3} \dd \tau.
\end{aligned}\ee
This implies
\be\begin{aligned}
&\|\langle x \rangle^{-N} (\tilde T(t, s) - \tilde T_2(t, s)) \langle x \rangle^{-N}\|_{2 \to 2} \leq \\
&\leq C \|\Omega'\|_{\infty} |t-s| \cdot \sup_{\tau \in [s, t]} \sup_{\omega} \|\langle x \rangle^{-N} \omega \Omega(\tau) \Omega(s)^{-1} \tilde T_2(t, s) \langle x \rangle^{-N}\|_{2 \to 2} \\
&= C \|\Omega'\|_{\infty} |t-s| \cdot \sup_{\tau \in [s, t]} \sup_{\omega} \|\langle x \rangle^{-N} \omega \tilde T_2(t, s) \langle x \rangle^{-N}\|_{2 \to 2},
\end{aligned}\ee
where the supremum is taken over all infinitesimal rotations $\omega$.

From the explicit form of the fundamental solution, we get that
\be
\sup_{\omega} |\omega T(t, s)|(x, y) \leq C \frac {|x| |y|}{|t-s|^{5/2}}.
\ee
Then,
\be\begin{aligned}
\sup_{\omega} |\omega \tilde T_2(t, s)|(x, y) &= \sup_{\omega} |\omega T(t, s)|(x + \tilde D(\tau) - \tilde D(s), y) \\
&\leq C \frac {(|x| + |D(\tau) - D(s)|) |y|}{|t-s|^{5/2}}.
\end{aligned}\ee
For sufficiently large $N$, this implies
\be\begin{aligned}\lb{2.178}
&\|\langle x \rangle^{-N} (\tilde T(t, s) - \tilde T_2(t, s)) \langle x \rangle^{-N}\|_{2 \to 2} \leq \\
&\leq C \big(|t-s|^{-3/2} \|\Omega'\|_{\infty} + |t-s|^{-1/2} \|\Omega'\|_{\infty} \|D'\|_{\infty}\big).
\end{aligned}\ee
By (\ref{3.167}), (\ref{2.169}), and (\ref{2.178}),
\be\begin{aligned}
&\|\langle x \rangle^{-N} (\tilde T_2(t, s) - T(t, s)) \langle x \rangle^{-N}\|_{2 \to 2} \leq \\
&\leq C \big(|t-s|^{-3/2} (\|\Omega'\|_{\infty} + \|D'\|_{\infty}) + |t-s|^{-1/2} (\|\Omega'\|_{\infty} \|D'\|_{\infty} + \|D'\|_{\infty}^2 + \|A'\|_{\infty})\big).
\end{aligned}\ee
This proves property P3.


Finally, P4 follows from the fact that eigenfunctions of $\mc H$ are in $\langle \dl \rangle^{-2} L^{6/5, 2}$, by Lemma \ref{pp}. Indeed, by (\ref{omega}),
\be
\partial_t U(t) U^{-1}(t) = |\Omega'(t)| \omega(t) + (\Omega(t) D'(t)) \dl + A'(t) \sigma_3,
\ee
where $\omega(t)$ is an infinitesimal rotation.

Then, let $f \in \langle \dl \rangle^{-2} L^{6/5, 2}$ solve, with $\lambda \not \in (-\infty, -\mu] \cup [\mu, \infty)$,
\be
f = R_0(\lambda) V f.
\ee
It is immediate that
\be
\|(\Omega(t) D'(t)) \dl f\| + \|A'(t) \sigma_3 f\|_{L^{6/5, 2}} \leq C (|D'| + |A'|) \|f\|_{\langle \dl \rangle^{-2} L^{6/5, 2}}.
\ee
If $V \in L^{3/2, 1}$, then $f \in \langle x \rangle^{-2} L^{6/5, 1}$, implying that $(\Omega(t) D'(t)) \dl f + A'(t) \sigma_3 f \in L^{6/5, 1}$ instead.

For the $\Omega(t)$ component of $U(t)$, P4 reduces to knowing that $\omega f \in L^{6/5, 2}$. However,
\be
|\omega R_0(\lambda)|(x, y) \leq C (1 + |x-y|) |R_0(\lambda)|(x, y).
\ee
P4 follows because $Vf \in L^{6/5, 2}$ and $|R_0(\lambda)|(x, y)$ decays exponentially.

If $V \in L^{3/2, 1}$, then $Vf \in L^{6/5, 1}$ implies that $\omega f \in L^{6/5, 1}$.
\end{proof}

\begin{proof}[Proof of Theorem \ref{theorem_13}] 
To begin with, assume that $V \in L^{\infty}$. This endows the Schr\"{o}dinger evolution with a meaning, in the space of exponentially growing functions in the $L^2$ norm. In the end, one can discard this assumption following an approximation argument.

Write the equation in the form (\ref{1.30}) and let
\be\begin{aligned}
&\tilde Z = P_c Z,\ \tilde F = P_c U(t) F - i[P_c, \partial_t U(t) U(t)^{-1}] (\tilde Z + P_p Z).\\
\end{aligned}\ee
The equation becomes
\be\begin{aligned}\lb{2.184}
&i \partial_t \tilde Z - i \partial_t U(t) U(t)^{-1} \tilde Z  + {\mc H} \tilde Z = \tilde F,\ \tilde Z(0) = P_c U(0) R(0) \text{ given}.
\end{aligned}\ee
Fixing $\delta>0$, we rewrite (\ref{2.184}) in the equivalent form
\be\begin{aligned}
&i \partial_t \tilde Z - i \partial_t U(t) U(t)^{-1} \tilde Z + (\mc H P_c +i \delta P_p) \tilde Z = \tilde F.
\end{aligned}\ee
Note that
\be
\mc H P_c +i \delta P_p = \mc H_0 + V - P_p (\mc H - i \delta).
\ee
Lemma \ref{lemma32} provides a decomposition
\be
V - P_p ({\mc H} - i\delta) = \tilde V_1 \tilde V_2,
\ee
where $\tilde V_1$ and $\tilde V_2^*$ are in $\mc L(L^2, L^{6/5, 2})$ and can be approximated in this space by operators $F_1^n$ and $F_2^n$, such that $F_1^n$, $(F_2^n)^* \in \mc L(L^2, \langle x \rangle^{-N} L^2)$.

Thus, we obtain
\be
i \partial_t \tilde Z - i \partial_t U(t) U(t)^{-1} \tilde Z + (\mc H_0 + \tilde V_1 \tilde V_2) \tilde Z = \tilde F.
\ee

As a general model, consider a solution of the equation
\be
i \partial_t f - i \partial_t U(t) U(t)^{-1} f + \mc H_0 f = F, f(0) \text{ given}.
\ee
Letting $f = U(t) g$, $F = U(t) G$, the equation becomes
\be
i \partial_t g + \mc H_0 g = G, g(0) = U(0)^{-1} f(0),
\ee
so
\be
g = e^{it \mc H_0} g(0) - i \int_{-\infty}^t e^{i(t-s) \mc H_0} G(s) \dd s
\ee
and
\be
f = e^{it \mc H_0} U(t) U(0)^{-1} f(0) - i \int_{-\infty}^t e^{i(t-s) \mc H_0} U(t) U(s)^{-1} F(s) \dd s.
\ee
In particular, we obtain that
\be\lb{2.193}
\tilde Z = \int_{-\infty}^t e^{i(t-s) \mc H_0} U(t) U(s)^{-1} (\delta_{s=0} \tilde Z(0) -i\tilde F(s) + i\tilde V_1 \tilde V_2 \tilde Z) \dd s.
\ee
Denote
\be\lb{3.144}\begin{aligned}
\tilde T_{\tilde V_2, \tilde V_1} F(t) &= \int_{-\infty}^t \tilde V_2 e^{i(t-s) {\mc H}_0} U(t) U(s)^{-1} \tilde V_1 F(s) \dd s, 
\end{aligned}\ee
respectively
\be\begin{aligned}
\tilde T_{\tilde V_2, I} F(t) &= \int_{-\infty}^t \tilde V_2 e^{i(t-s) {\mc H}_0} U(t) U(s)^{-1} F(s) \dd s.
\end{aligned}\ee
Then, rewrite Duhamel's formula (\ref{2.193}) as
\be\lb{3.186}\begin{aligned}
(I - i \tilde T_{\tilde V_2, \tilde V_1}) \tilde V_2 \tilde Z(t) &= \tilde T_{\tilde V_2, I} (-i \tilde F(s) + \delta_{s=0} \tilde Z(0)).
\end{aligned}\ee
We compare $\tilde T_{\tilde V_2, \tilde V_1}$ with the kernel, invariant under time translation,
\be\begin{aligned}
T_{\tilde V_2, \tilde V_1} F(t) &= \int_{-\infty}^t \tilde V_2 e^{i(t-s) {\mc H}_0} \tilde V_1 F(s) \dd s. 
\end{aligned}\ee
Indeed, $I - i T_{\tilde V_2, \tilde V_1}$ is invertible, see (\ref{2.209}), and we want to prove the same for $\tilde T$, so that we can invert in (\ref{3.186}).

We make the comparison in the following algebra $\tilde K$, whose definition parallels that of $K$, Definition \ref{def_k}.
\begin{definition}\lb{def_kt}
$
\tilde K = \{T(t, s) \mid \sup_s \|T(t, s) f\|_{M_t L^2_x} \leq C \|f\|_2\}.
$
\end{definition}


Here $M_t L^2_x$ is the set of $L^2$-valued measures of finite mass on the Borel algebra of $\set R$. $\tilde K$ contains operators that are not translation-invariant in $t$ and $s$, is naturally endowed with a unit, but is not a $C^*$ algebra.

To carry out the comparison of $T$ and $\tilde T$, let
\be
f(\pi, \tau) = \sup_{t-s=\tau} \|\langle x \rangle^{-N} (U(t) U(s)^{-1} e^{i(t-s) \mc H_0} - e^{i(t-s) \mc H_0}) \langle x \rangle^{-N}\|_{2 \to 2}
\ee
and observe that
\be
\|\langle x \rangle^{-N} (U(t) U(s)^{-1} e^{i(t-s) \mc H_0} - e^{i(t-s) \mc H_0}) \langle x \rangle^{-N}\|_{\tilde K} \leq \int_{\set R} f(\pi, \tau) \dd \tau.
\ee
By condition P3, for almost every $\tau$
\be\begin{aligned}
\lim_{\substack{\|\pi'\|_{\infty} \to 0}} f(\pi, \tau) = 0.
\end{aligned}\ee
On the other hand, by combining the $L^2$ boundedness and the $L^1 \to L^{\infty}$ decay estimates,
\be\begin{aligned}\lb{3.172}
&\|\langle x \rangle^{-N} (U(t) U(s)^{-1} e^{i(t-s) \mc H_0} - e^{i(t-s) \mc H_0}) \langle x \rangle^{-N}\|_{2 \to 2} \leq \\
&\leq \|\langle x \rangle^{-N} U(t) U(s)^{-1} e^{i(t-s) \mc H_0} \langle x \rangle^{-N}\|_{2 \to 2} + \|\langle x \rangle^{-N} e^{i(t-s) \mc H_0} \langle x \rangle^{-N}\|_{2 \to 2}\\
&\leq C \min (1, |t-s|^{-3/2}).
\end{aligned}\ee
Equivalently,
\be
|f(\pi, \tau)| \leq C \min (1, |t-s|^{-3/2}),
\ee
so $f(\pi, \tau)$ is uniformly integrable in $\tau$. By dominated convergence,
\be
\lim_{\|\pi'\|_\infty \to 0} \int_{\set R} f(\pi, \tau) \dd \tau = 0,
\ee
hence
\be
\lim_{\substack{\|\pi'\|_{\infty} \to 0}} \|\langle x \rangle^{-N} (U(t) U(s)^{-1} e^{i(t-s) \mc H_0} - e^{i(t-s) \mc H_0}) \langle x \rangle^{-N}\|_{\tilde K} = 0.
\ee
Therefore, for each approximation $F_1^n$ and $F_2^n$ of $\tilde V_1$ and $\tilde V_2$ 
\be
\lim_{\substack{\|\pi'\|_{\infty} \to 0}} \|T_{F_2^n, F_1^n} - \tilde T_{F_2^n, F_1^n}\|_{\tilde K} = 0.
\ee
Since $F_2^n$ and $F_1^n$ are approximations of $\tilde V_2$ and $\tilde V_1$, by (\ref{1.45})
\be\begin{aligned}\lb{2.166}
&\lim_{n \to \infty} \|T_{F_2^n, F_1^n} - T_{\tilde V_2, \tilde V_1}\|_{\mc L(L^2_{t, x}, L^2_{t, x})} = 0, \\
&\lim_{n \to \infty} \|\tilde T_{F_2^n, F_1^n} - \tilde T_{\tilde V_2, \tilde V_1}\|_{\mc L(L^2_{t, x}, L^2_{t, x})} = 0.
\end{aligned}\ee
When $V \in L^{3/2, 1}$, one can replace $\mc L(L^2_{t, x}, L^2_{t, x})$ by $\tilde K$ in (\ref{2.166}), using (\ref{1.46}) instead of (\ref{1.45}).

Therefore, when $V \in L^{3/2, \infty}_0$
\be\lb{3.189}
\lim_{\substack{\|\pi'\|_{\infty} \to 0}} \|T_{\tilde V_2, \tilde V_1} - \tilde T_{\tilde V_2, \tilde V_1}\|_{\mc L(L^2_{t, x}, L^2_{t, x})} = 0
\ee
and when $V \in L^{3/2, 1}$
\be
\lim_{\substack{\|\pi'\|_{\infty} \to 0}} \|T_{\tilde V_2, \tilde V_1} - \tilde T_{\tilde V_2, \tilde V_1}\|_{\tilde K} = 0.
\ee
However, the operator $I - i T_{\tilde V_1, \tilde V_2}$ is invertible in $\mc L(L^2_{t, x}, L^2_{t, x})$. Indeed, its inverse is formally (say, for compactly supported functions in $t$) given by the following formula:
\be\begin{aligned}\lb{2.209}
(I - i T_{\tilde V_1, \tilde V_2})^{-1} F(t) &= F(t) - i\int_{-\infty}^t \tilde V_2 e^{i(t-s)\mc H P_c - \delta(t-s) P_p} \tilde V_1 F(s) \dd s.
\end{aligned}\ee
Due to the Strichartz estimates of Theorem \ref{theorem_26}, this formal inverse actually is in $\mc L(L^2_{t, x}, L^2_{t, x})$.

For the same reason, when $V \in L^{3/2, 1}$, $I - i T_{\tilde V_1, \tilde V_2}$ is also invertible in the $\tilde K$ algebra of Definition \ref{def_kt}.

By (\ref{3.189}), it follows that when $\|\pi'\|_{\infty}$ is small enough $I - i \tilde T_{\tilde V_2, \tilde V_1}$ is also invertible in $\mc L(L^2_{t, x}, L^2_{t, x})$, respectively $\tilde K$.

Using the invertibility of $I - i \tilde T_{\tilde V_2, \tilde V_1}$ in (\ref{3.186}) results in
\be\begin{aligned}
\|\tilde V_2 \tilde Z\|_{L^2_{t, x}} &\leq C \|T_{\tilde V_2, I}(-i \tilde F + \delta_{s=0} \tilde Z(0))\|_{L^2_{t, x}} \\
&\leq C\big(\|\tilde F\|_{L^1_t L^2_x + L^2_t L^{6/5, 2}_x} + \|\tilde Z(0)\|_{L^2_x}\big),
\end{aligned}\ee
respectively
\be\begin{aligned}
\|\tilde V_2 \tilde Z\|_{L^1_t L^2_x} &\leq C \|T_{\tilde V_2, I}(-i \tilde F + \delta_{s=0} \tilde Z(0))\|_{L^1_t L^2_x} \\
&\leq C\big(\|\tilde F\|_{L^1_t L^{6/5, 1}_x} + \|\tilde Z(0)\|_{L^{6/5, 1}_x}\big).
\end{aligned}\ee
Plugging this back into Duhamel's formula (\ref{2.193}) leads to the desired estimates.

Finally, we account for the right-hand side term $\partial_t [P_c, \partial_t U(t) U(t)^{-1}] \tilde Z$. Property P4 and the structure of $P_c$ (see (\ref{pc}) and Lemma \ref{pp}) imply
\be
\|\partial_t [P_c, \partial_t U(t) U(t)^{-1}] \tilde Z\|_{L^2_t L^{6/5, 2}_x} \leq C \|\pi'\|_{\infty} \|\tilde Z\|_{L^2_t L^{6, 2}_x}.
\ee
Thus, the commutator $\partial_t [P_c, \partial_t U(t) U(t)^{-1}] \tilde Z$ is 
controlled by Strichartz inequalities and a fixed point argument in $L^2_t L^{6, 2}_x$, if $\|\pi'\|_{\infty}$ is small.

The same goes in regard to the proof of (\ref{2.157}), where the commutation term is controlled by P4 and a fixed point argument in $L^1_t L^{6, 1}_x$.
%
\end{proof}

\section*{Acknowledgments} I would like to thank Michael Goldberg,  Wilhelm Schlag, and the anonymous referee for their useful comments.

\appendix
\section{Atomic decomposition of Lorentz spaces}\lb{spaces} Our computations take place in Lebesgue and Sobolev spaces of functions defined on $\set R^{3+1}$. This  corresponds to three spatial dimensions and one extra dimension that accounts for time.


We denote the Lebesgue norm of $f$ by by $\|f\|_p$. The Sobolev norms of integral order $W^{n, p}$ are then defined by
\be
\|f\|_{W^{n, p}} = \bigg(\sum_{|\alpha| \leq n} \|\partial^{\alpha} f\|_p^p\bigg)^{1/p}
\ee
for $1 \leq p < \infty$ and
\be
\|f\|_{W^{n, \infty}} = \sup_{|\alpha| \leq n} \|\partial^{\alpha} f\|_{\infty}.
\ee
when $p=\infty$. In addition, we consider Sobolev spaces of fractional order, both homogenous and inhomogenous:
\be
\|f\|_{W^{s, p}} = \|\langle \dl \rangle^s f\|_p, \text{ respectively } \|f\|_{\dot W^{s, p}} = \||\dl| f\|_p.
\ee
Here $\langle \dl \rangle^s$ and $|\dl|^s$ denote Fourier multipliers --- multiplication on the Fourier side by $\langle \xi \rangle^s = (1+|\xi|^2)^{s/2}$ and $|\xi|^{s}$ respectively.

When $p=2$, the alternate notation $H^s = W^{s, 2}$ or $\dot H^s = \dot W^{s, 2}$ is customary.


In addition, we are naturally led to consider Lorentz spaces. Given a measurable function $f$ on a measure space $(X, \mu)$, consider its distribution function
\be
m(\sigma, f) = \mu(\{x \mid |f(x)|>\sigma\}).
\ee
\begin{definition}
A measurable function $f$ belongs to the Lorentz space $L^{p, q}$ if its decreasing rearrangement
\be\lb{A6}
f^*(t) = \inf\{\sigma \mid m(\sigma, f) \leq t\}
\ee
fulfills
\be
\|f\|_{L^{p, q}} = \bigg(\int_0^{\infty} (t^{1/p} f^*(t))^q \frac {\dd t} t \bigg)^{1/q} < \infty
\lb{lorentz}
\ee
or, respectively,
\be\lb{A7}
\|f\|_{L^{p, \infty}} = \sup_{0 \leq t < \infty} t^{1/p} f^*(t) < \infty
\ee
when $q=\infty$.
\end{definition}
We list several important properties of Lorentz spaces.
\begin{lemma}\begin{enumerate}\item[i)] $L^{p, p} = L^p$ and $L^{p, \infty}$ is weak-$L^p$.
\item[ii)] The dual of $L^{p, q}$ is $L^{p', q'}$, where $1/p + 1/p' = 1$, $1/q + 1/q' = 1$.
\item[iii)] If $q_1 \leq q_2$, then $L^{p, q_1} \subset L^{p, q_2}$.
\item[iv)] Except when $q=\infty$, the set of bounded compactly supported functions is dense in $L^{p, q}$.
\end{enumerate}\end{lemma}
For a more complete enumeration, see \cite{bergh}.



We conclude by proving a lemma concerning the atomic decomposition of Lorentz spaces. In preparation for that, we give the following definition:
\begin{definition}
The function $a$ is an $L^p$ atom, $1 \leq p < \infty$, if $a$ is measurable, bounded, its support has finite measure, and $a$ is $L^p$ normalized:
\be
\esssup_x |a(x)| < \infty,\ \mu(\supp a)< \infty,\ (\esssup_x |a(x)|)^p \cdot \mu(\supp a) = 1.
\ee
\end{definition}
Again, note that $a$ is an atom if and only if $|a|$ is one.

We call a sequence lacunary if the quotient between two successive elements is bounded from below by a number greater than one.
\begin{lemma}[Atomic decomposition of $L^{p, q}$]\lb{lemma_30} Consider a measure space $(X, \mu)$. A function $f$ belongs to $L^{p, q}(X)$, $1 \leq p$, $q < \infty$, if and only if it is a linear combination of $L^p$ atoms
\be\lb{A.9}
f = \sum_{k \in \set Z} \alpha_k a_k,
\ee
where the atoms $a_k$ have disjoint supports of lacunary sizes and $(\alpha_k)_k \in \ell^q(\set Z)$. Furthermore, $\|f\|_{L^{p, q}} \sim \|\alpha_k\|_{\ell^q}$ and the sum (\ref{A.9}) converges unconditionally in the $L^{p, q}$ norm.
\end{lemma}
In the limiting case $L^{p, \infty}$, a similar decomposition exists, but only converges when $f$ is in $L^{p, \infty}_0$ --- the closure within $L^{p, \infty}$ of the set of bounded functions of finite-measure support.

Note that, for $1<p<\infty$, $L^{p', 1}$ is the dual of $L^{p, \infty}_0$ and $L^{p, \infty}$ is the dual of $L^{p', 1}$, but the dual of $L^{p, \infty}$ can be quite complicated.
\begin{proof}
In one direction, assume that $f \in L^{p, q}$; without loss of generality, we may consider $|f|$ in its stead. 


Let $f_k=f^*(2^k)$, for $f^*$ as in (\ref{A6}). Since the distribution function is decreasing, by (\ref{lorentz}) one has that
\be
\bigg(\sum_k (2^{(k+1)/p} f_k)^q \bigg)^{1/q} \geq \|f\|_{L^{p, q}} \geq \bigg(\sum_k (2^{k/p} f_{k+1})^q \bigg)^{1/q}
\ee
or, equivalently,
\be
2^{1/p} \bigg(\sum_k (2^{k/p} f_k)^q \bigg)^{1/q} \geq \|f\|_{L^{p, q}} \geq 2^{-1/p} \bigg(\sum_k (2^{k/p} f_k)^q \bigg)^{1/q}.
\ee
Thus, $\big(\sum_k (2^{k/p} f_k)^q \big)^{1/q}$ is comparable to $\|f\|_{L^{p, q}}$.

Let $B_k = |f|^{-1}([f_k, \infty))$. By definition, $\mu(B_k) \leq 2^k$ and $|f(x)| \geq f_k$ on $B_k$. Observe that if $\mu(B_k) > \mu(B_{k-1})$, then $\mu(B_k) > 2^{k-1}$.

For any $k \in \set Z$, let
\be
n(k)=\min\{n \mid \mu(B_n \setminus B_k) \geq 2^k\},\ m(k)=\max\{m \mid \mu(B_k \setminus B_m) \geq 2^{k-1}\},
\ee
and set $n(k) = +\infty$ or $m(k) = -\infty$ if the sets are empty. Observe that, for $k \leq \ell < n(k)$, $f_{\ell} = f_k$ and, for $k > \ell \geq m(k)$, $f_{\ell} = f_{k-1}$.

Define recursively the finite or infinite sequence $(k_{\ell})_{\ell \in \set Z}$ by
\be
k_0 = 1,\ k_{\ell+1} = n(k_{\ell})\ \text{for } \ell \geq 0,\ \text{and } k_{\ell-1} = m(k_{\ell})\ \text{for } \ell \leq 0.
\ee
Then $f_k = f_{k_{\ell}}$ whenever $k_{\ell} \leq k < k_{\ell+1}$. Since
\be
2^{-1/p} 2^{k_{\ell+1}/p} \leq \sum_{k=k_{\ell}}^{k_{\ell+1}-1} 2^{k/p} \leq \frac {2^{-1/p}}{1-2^{-1/p}} 2^{k_{\ell+1}/p},
\ee
it follows that $\big(\sum_{\ell} (2^{k_{\ell+1}/p} f_{k_{\ell}})^q \big)^{1/q}$ is comparable to $\|f\|_{L^{p, q}}$.

Let $A_{\ell} = |f|^{-1}([f_{k_{\ell}}, f_{k_{\ell-1}}))$. Note that $\mu(A_{\ell}) \geq 2^{k_{\ell-1}}$ by the definition of the $k_{\ell}$ sequence and also that $\mu(A_{\ell}) = \mu(B_{\ell} \setminus B_{\ell-1} > 2^{k_{\ell}-1}-2^{k_{\ell-1}}$. We infer that $\mu(A_{\ell}) > 2^{k_{\ell}-2}$. On the other hand, $\mu(A_{\ell}) \leq \mu(B_{\ell}) \leq 2^{k_{\ell}}$.

Set
\be\begin{aligned}
\alpha_{\ell} &= \mu(A_{\ell})^{1/p} \esssup\{|f(x)| \mid x \in A_{\ell}\}, \\
a_{\ell} &= (\chi_{A_{\ell}} f)/\alpha_{\ell}.
\end{aligned}\ee
Note that $a_{\ell}$ is an atom for each $\ell$ and
\be
2^{k_{\ell}/p} f_{k_{\ell-1}} > \alpha_{\ell} > 2^{(k_{\ell}-2)/p} f_{k_{\ell}}.
\ee
In particular,
\be
\sum_{\ell} \alpha_{\ell}^q < \sum_{\ell} 2^{k_{\ell}/p} f_{k_{\ell-1}} \leq C \|f\|_{L^{p, q}},
\ee
so $\|\alpha_k\|_{\ell^q} \leq C \|f\|_{L^{p, q}}$.

In order to establish the converse, consider
\be
f = \sum_{k \in \set Z} \alpha_k a_k,
\ee
where $a_k$ are $L^p$ atoms of disjoint supports of size $2^k$ and only finitely many of the coefficients $\alpha_k$ are nonzero. One needs to show that
\be\lb{A.17}
\|f\|_{L^{p, q}} \leq C \|\alpha_k\|_{\ell^q},
\ee
with a constant that does not depend on the number of terms.

Observe that each atom $a_k$ has $\esssup |a_k(x)| = 2^{-k/p}$. Since the measures of supports of atoms $a_k$, for $k < k_0$, add up to at most $2^{k_0}$, it follows by the definition of the distribution function that
\be
f^*(2^{k_0}) \leq \esssup_{k \geq k_0,\ x \in X} |\alpha_k a_k(x)| = \sup_{k \geq k_0}2^{-k/p} |\alpha_k|.
\ee
Then, the integral that appears in the definition (\ref{lorentz}) can be bounded by
\be\begin{aligned}
\int_0^{\infty} (t^{1/p} f^*(t))^q \frac {\dd t} t &\leq \sum_{k_0 \in \set Z} \big(2^{k_0/p} (\sup_{k \geq k_0}2^{-k/p} |\alpha_k|)\big)^q \\
&\leq \sum_{k \in \set Z} \big(\sum_{k_0 \leq k} 2^{k_0 q/p}\big) 2^{-kq/p} |\alpha_k|^q \\
&= \frac {2^{q/p}}{2^{q/p}-1} \sum_{k \in \set Z} |\alpha_k|^q.
\end{aligned}\ee
It follows that $\|f\|_{L^{p, q}}$ is indeed finite and fulfills (\ref{A.17}). The unconditional convergence of $\sum_k |\alpha_k|^q$ implies the unconditional convergence of $\sum \alpha_k a_k$.

This proof of the converse easily generalizes to the case when the atom sizes are only in the order of, not precisely equal to, $2^k$, as well as to the case when there are at most some fixed number of atoms of each dyadic size instead of just one (in particular, to the case of a lacunary sequence).
\end{proof}

A characterization of $L^{p, \infty}$ is given by (\ref{A7}) or, equivalently,
\be
L^{p, \infty} = \{f \mid \sup_t t^p \mu(\{x \mid |f(x)| > t\}) < \infty\}.
\ee
Next, we characterize $L^{p, \infty}_0$, the closure in $L^{p, \infty}$ of the set of bounded functions of finite support, in a similar manner.
\begin{proposition} \lb{prop_a3} For $1 \leq p < \infty$, $L^{p, \infty}_0 \subset L^{p, \infty}$ is characterized by
\be\lb{A.23}\begin{aligned}
L^{p, \infty}_0 = \{f \in L^{p, \infty} \mid \lim_{t \to \infty} t^p \mu(\{x \mid |f(x)| > t\}) = 0,\\
\lim_{t \to 0} t^p \mu(\{x \mid |f(x)| > t\}) = 0\}.
\end{aligned}\ee
\end{proposition}
\begin{proof} If $f \in L^p$, then
\be
\int_0^{\infty} t^p \dd \mu(\{x \mid |f(x)| > t\}) < \infty,
\ee
implying (\ref{A.23}). Since $L^p$ is dense in $L^{p, \infty}_0$, the same follows for $L^{p, \infty}_0$.

Conversely, (\ref{A.23}) implies that $f$ can be approximated by
\be
f_n(x) = f(x) \chi_{\{x \mid 1/n < |f(x)| < n\}}(x)
\ee
and $f_n$ are bounded and of compact support.
\end{proof}
%


\section{Real interpolation}

Our presentation of interpolation follows Bergh--L\"{o}f\-str\"{o}m, \cite{bergh}, and is included only for the sake of completeness. The reader is advised to consult this reference work for a much more detailed exposition.

For any couple of Banach spaces $(A_0, A_1)$, contained within a wider space $X$, their intersection, $A_0 \cap A_1$, and their sum
\be
A_0+A_1 = \{x \in X \mid x = a_0+a_1,\ a_0 \in A_0,\ a_1 \in A_1\}
\ee
give rise to two potentially new Banach spaces.

Given a couple of Banach spaces $(A_0, A_1)$ as above, define the so-called $K$ functional --- method due to Peetre --- on $A_0+A_1$ by
\be
K(t, a) = \inf_{a = a_0 + a_1} (\|a_0\|_{A_0} + t \|a_1\|_{A_1}).
\ee
\begin{definition}
For $0 \leq \theta \leq 1$, $1 \leq q \leq \infty$, the interpolation space $(A_0, A_1)_{\theta, q} = A_{\theta, q}$ is the set of elements $f \in A_0 + A_1$ whose norm
\be
\|f\|_{A_{\theta, q}} = \bigg(\int_0^{\infty} (t^{-\theta} K(t, f))^q \frac {dt} t \bigg)^{1/q}
\ee
is finite.
\end{definition}

$(A_0, A_1)_{\theta, q}$ is an exact interpolation space of exponent $\theta$ between $A_0$ and $A_1$,  meaning that it satisfies the following defining property:
\begin{theorem} Let $T$ be a bounded linear mapping between two pairs of Banach spaces $(A_0, A_1)$ and $(B_0, B_1)$, i.e.
\be
\|T f\|_{B_j} \leq M_j \|f\|_{A_j},\ j = 0, 1.
\ee
Then
\be
\|T f\|_{(B_0, B_1)_{\theta, q}} \leq M_0^{1-\theta} M_1^{\theta} \|f\|_{(A_0, A_1)_{\theta, q}}.
\ee
\end{theorem}

For two couples of Banach spaces, $(A_0^{(1)}, A_1^{(1)})$ and $(A_0^{(2)}, A_1^{(2)})$,
\be
(A_0^{(1)} \times A_0^{(2)}, A_1^{(1)} \times A_1^{(2)})_{\theta, q} = (A_0^{(1)}, A_1^{(1)})_{\theta, q} \times (A_0^{(2)}, A_1^{(2)})_{\theta, q}.
\ee

The following multilinear interpolation theorem is due to Lions--Peetre:



\begin{theorem}\lb{thm_32} Assume that $T$ is a bilinear mapping from $(A_0^{(j)} \times A_1^{(j)})$ to $B^j$ for $j = 0, 1$ and
\be
\|T(a^{(1)}, a^{(2)})\|_{B_j} \leq M_j \|a^{(1)}\|_{A_j^{(1)}} \|a^{(2)}\|_{A_j^{(2)}}.
\ee
Then
\be
\|T(a^{(1)}, a^{(2)})\|_{(B_0, B_1)_{\theta, q}} \leq C \|a^{(1)}\|_{(A_0^{(1)}, A_1^{(1)})_{\theta, q_1}} \|a^{(2)}_j\|_{(A_0^{(2)}, A_1^{(2)})_{\theta, q_2}},
\ee
if $0<\theta<1$, $1/q-1=(1/q_1-1) + (1/q_2-1)$, $1 \leq q \leq \infty$.
\end{theorem}

Another bilinear real interpolation theorem is the following (see \cite{bergh}, Section 3.13.5(b)):
\begin{theorem}\lb{int_keta}
Let $A_0$, $A_1$, $B_0$, $B_1$, $C_0$, and $C_1$ be Banach spaces, 
and assume that the bilinear operator $T$ is bounded as follows:
\be
T : A_0 \times B_0 \mapsto C_0,\  T : A_0 \times B_1 \mapsto C_1,\ T : A_1 \times B_0 \mapsto C_1.
\ee 
If $0 < \theta_0, \theta_1 < \theta < 1$, $1 \leq a, b, r \leq \infty$ and $1 \leq \frac 1 a + \frac 1 b$, $\theta = \theta_0 + \theta_1$, then
\be
T : (A_0, A_1)_{\theta_1, ar} \times (B_0, B_1)_{\theta_2, br} \to (C_0, C_1)_{\theta, r}.
\ee
\end{theorem}

Below we list the results of real interpolation in some standard situations:
\begin{proposition}\label{prop_33} Let $L^p$ be the Lebesgue spaces defined over a measure space $(X, \mu)$ and $L^{p, q}$ be Lorentz spaces over the same. Then
\begin{enumerate}
\item[1.] $(L^{p_0, q_0}, L^{p_1, q_1})_{\theta, q} = L^{p, q}$, for $p_0$, $p_1$, $q_0$, $q_1 \in (0, \infty]$, $p_0 \ne p_1$, $1/p = (1-\theta)/p_0 + \theta/p_1$, $0<\theta<1$.
\item[2.] For a Banach space $A$, define the weighted spaces of sequences (homogenous and inhomogenous, respectively)
\be\begin{aligned}
\dot \ell^q_s(A) &= \bigg\{(a_n)_{n \in \set Z} \mid \|(a_n)\|_{\ell^q_s(A)} = \bigg(\sum_{n=-\infty}^{\infty} (2^{ns} \|a_n\|_A)^q\bigg)^{1/q} < \infty\bigg\}\\
\ell^q_s(A) &= \bigg\{(a_n)_{n \geq 0} \mid \|(a_n)\|_{\ell^q_s(A)} = \bigg(\sum_{n=0}^{\infty} (2^{ns} \|a_n\|_A)^q\bigg)^{1/q} < \infty\bigg\}
\end{aligned}\ee
Then
\be\begin{aligned}
(\dot \ell_{s_0}^{q_0}(A_0), \dot \ell_{s_1}^{q_1}(A_1))_{\theta, q} &= \dot \ell_s^q((A_0, A_1)_{\theta, q}), \\
(\ell_{s_0}^{q_0}(A_0), \ell_{s_1}^{q_1}(A_1))_{\theta, q} &= \ell_s^q((A_0, A_1)_{\theta, q}),
\end{aligned}\ee
where $0<q_0$, $q_1<\infty$, $s = (1-\theta) s_0 + \theta s_1$, $1/q = (1-\theta)/q_0 + \theta/q_1$, $0<\theta<1$.
\item[3.] $(L^{p_0}(A_0), L^{p_1}(A_1))_{\theta, p} = L^p((A_0, A_1)_{\theta, p})$, for $1/p = (1-\theta)/p_0 + \theta/p_1$.
\end{enumerate}\end{proposition}
Again, the reader is referred to \cite{bergh} for the proofs and for more details.

Finally, real interpolation yields a short proof of sharpened Young's and H\"{o}lder's inequalities that we use throughout the paper:
\begin{proposition}\lb{prop_holder}
Assume that $f \in L^{p_1, q_1}$ and $g \in L^{p_2, q_2}$, $1\leq p_1, q_1, p_2, q_2 \leq \infty$. If $\frac 1 p = \frac 1 {p_1} + \frac 1 {p_2}$, $\frac 1 q = \frac 1 {q_1} + \frac 1 {q_2}$, $1<p_1, p_2<\infty$, then $f\, g \in L^{p, q}$.

If $\frac 1 {\tilde p} = \frac 1 {p_1} + \frac 1 {p_2} - 1$ and $1 <p_1, p_2, p < \infty$, then $f * g \in L^{\tilde p, q}$.
\end{proposition}
\begin{proof}
Interpolate by Theorem \ref{int_keta} between
\be
\|f g\|_{L^{\infty}} \leq \|f\|_{L^{\infty}} \|g\|_{L^{\infty}},\ \|f g\|_{L^1} \leq \|f\|_{L^{\infty}} \|g\|_{L^1},\ \text{and } \|f g\|_{L^1} \leq \|f\|_{L^1} \|g\|_{L^{\infty}},
\ee
with interpolation exponents $\theta_0=\frac 1 {p_1}$, $\theta_1=\frac 1 {p_2}$, $\theta=\frac 1 p$, $ar = q_1$, $br = q_2$, and $r=q$. We obtain that $f\, g \in L^{p, q}$.

Concerning Young's inequality, interpolate by Theorem \ref{int_keta} between
\be
\|f * g\|_{L^1} \leq \|f\|_{L^1} \|g\|_{L^1},\ \|f * g\|_{L^{\infty}} \leq \|f\|_{L^{\infty}} \|g\|_{L^1},\ \text{and } \|f * g\|_{L^{\infty}} \leq \|f\|_{L^1} \|g\|_{L^{\infty}},
\ee
with interpolation exponents $\theta_0 = 1 - \frac 1 {p_1}$, $\theta_1 = 1 - \frac 1 {p_2}$, $\theta = 2 - \frac 1 {p_1} - \frac 1 {p_2} = 1 - \frac 1 {\tilde p}$, $ar = q_1$, $br = q_2$, and $r=q$. We obtain exactly that $f * g \in L^{\tilde p, q}$.
\end{proof}



\begin{thebibliography}{AbcDef1}\addcontentsline{toc}{TOCspecial}{References}
\bibitem[Agm]{agmon} S.\ Agmon, \emph{Spectral properties of Schr\"{o}dinger operators and scattering theory}, Ann.\ Scuola Norm.\ Sup.\ Pisa Cl.\ Sci.\ (4)  2 (1975), No.\ 2, pp.\ 151--218.


\bibitem[Bec1]{bec} M.\ Beceanu, \emph{A centre-stable manifold for the focusing cubic NLS in $\set R^{1+3}$}, Comm.\ Math.\ Phys.\ 280 (2008), no.\ 1, pp.\ 145--205.

\bibitem[Bec2]{bec2} M.\ Beceanu, \emph{A critical centre-stable manifold for the semilinear Schr\"{o}dinger equation}, submitted. 

\bibitem[Bec3]{bec3} M.\ Beceanu, \emph{Structure of wave operators for critical potentials}, in preparation.

\bibitem[BeGo]{bego} M.\ Beceanu, \emph{Decay estimates for the Schr\"{o}dinger equation with critical potentials}, in preparation.



\bibitem[BeL\"o]{bergh} J.\ Bergh, J.\ L\"ofstr\"om, \emph{Interpolation Spaces. An Introduction}, Springer-Verlag, 1976.

\bibitem[Bou1]{bou1} J.\ Bourgain, \emph{On long-time behaviour of solutions of linear Schr\"{o}dinger equations with smooth time-dependent potential}, Geometric aspects of functional analysis, pp.\ 99--113, Lecture Notes in Math., 1807, Springer, Berlin, 2003.

\bibitem[Bou2]{bou2} J.\ Bourgain, \emph{Growth of Sobolev norms in linear Schr\"{o}dinger equations with quasi-periodic potential}, Comm.\ Math.\ Phys.\ 204, no.\ 1, pp.\ 207--240, 1999.

\bibitem[Bou3]{bou3} J.\ Bourgain, \emph{On growth of Sobolev norms in linear Schr\"{o}dinger equations with smooth time-dependent potential}, J.\ Anal Math.\ 77, pp.\ 315--348 (1999).

\bibitem[Bou4]{bou4} J.\ Bourgain, \emph{Fourier transform restriction phenomena for certain lattice subsets and applications to nonlinear evolution equations. I.\ Schr\"{o}dinger equations}, Geom.\ Funct.\ Anal 3, no.\ 2, pp. 107--156 (1993). 

\bibitem[BPST1]{bpst1} N.\ Burq, F.\ Planchon, J.\ S.\ Stalker, A.\ S.\ Tahvildar-Zadeh, \emph{Strichartz estimates for the wave and Schr\"{o}dinger equations with potentials of critical decay}, Indiana Univ.\ Math.\ J.\ 53  (2004), No.\ 6, pp.\ 1665--1680.

\bibitem[BPST2]{bpst2} N.\ Burq, F.\ Planchon, J.\ S.\ Stalker, A.\ S.\ Tahvildar-Zadeh, \emph{Strichartz estimates for the wave and Schr\"{o}dinger equations with the inverse-square potential}, J.\ Funct.\ Anal.\ 203 (2003), No.\ 6, pp.\ 519--549.

\bibitem[CLT]{clt} O.\ Costin, J.\ L.\ Lebowitz, S.\ Tanveer, \emph{Ionization of Coulomb systems in $\set R^3$ by time periodic forcing of arbitrary size}, Comm.\ Math.\ Phys., 296 (3), pp. 681--738 (2010).









\bibitem[CuMi]{cucmiz} S. Cuccagna, T.\ Mizumachi, \emph{On asymptotic stability in energy space of ground states for Nonlinear Schr\"odinger equations}, preprint 2007.


\bibitem[Del]{delort} J.-M. Delort, \emph{Normal forms and long time existence for semi-linear Klein-Gordon equations}, Boll.\ Unione Mat.\ Ital.\ Sez.\ B Artic.\ Ric.\ Mat.\ (8) 10 (2007), no.\ 1, pp.\ 1--23.





\bibitem[ErSc]{erdsch2} B.\ Erdo\^{g}an, W.\ Schlag, \emph{Dispersive estimates for Schr\"odinger operators in the presence of a resonance and/or an eigenvalue at zero energy in dimension three: II}, Journal d'Analyse Math\'{e}matique, Vol.\ 99, No.\ 1 (2006), pp.\ 199--248.

\bibitem[Fos]{foschi} D.\ Foschi, \emph{Inhomogeneous Strichartz estimates}, J.\ Hyperbolic Differ.\ Equ.\ 2 (1) (2005) 1 24.




\bibitem[GJY]{gjy} A.\ Galtbayar, A.\ Jensen, K.\ Yajima, \emph{Local time-decay of solutions to Schr\"{o}dinger equations with time-periodic potentials}, J.\ Stat.\ Phys. 1--4, pp.\ 283--310 (2004). 

\bibitem[Gol1]{gol} M. Goldberg, \emph{Strichartz estimates for the Schr\"{o}dinger equation with time-periodic $L^{n/2}$ potentials}, Journal of Functional Analysis, pp. 718--746, 256, 3, (2009).

\bibitem[Gol2]{gol2} M.\ Goldberg, \emph{Dispersive Bounds for the Three-Dimensional Schršdinger Equation with Almost Critical Potentials}, Geom.\ and Funct.\ Anal.\ 16 (2006), no.\ 3, pp.\ 517--536.

\bibitem[GoSc]{golsch} M.\ Goldberg, W.\ Schlag, \emph{A limiting absorption principle for the three-dimensional Schr\"odinger equation with $L^p$ potentials}, Intl.\ Math.\ Res.\ Not.\ 2004:75 (2004), pp.\ 4049--4071.




\bibitem[How]{how} J.\ Howland, \emph{Stationary scattering theory for time-dependent Hamiltonians}, Mathematische Annalen, Volume 207, Number 4 / December, 1974, pp.\ 315--335.

\bibitem[HuLe]{hule} D.\ Hundertmark, Y.-R.\ Lee, \emph{Exponential decay of eigenfunctions and generalized eigenfunction of non self-adjoint matrix Schršdinger operators related to NLS}, Bulletin of the London Mathematical Society, 39 (2007), no.\ 5, pp.\ 709--720.


\bibitem[IoJe]{ionjer} A.\ D.\ Ionescu, D.\ Jerison, \emph{On the absence of positive eigenvalues of Schr\"odinger operators with rough potentials}, Geometric and Functional Analysis 13, pp.\ 1029--1081 (2003).

\bibitem[JeKa]{jeka} A.\ Jensen, T.\ Kato, \emph{Spectral properties of Schr\"{o}dinger operators and time-decay of the wave functions}, Duke Math.\ J.\ Volume 46, Number 3 (1979), pp.\ 583--611. 

\bibitem[JSS]{jss} J.-L.\ Journ\'{e}, A.\ Soffer, C.\ D.\ Sogge, \emph{Decay estimates for Schr\"{o}dinger operators}, Comm.\ Pure Appl.\ Math.\ 44 (1991), no.\ 5, pp.\ 573--604.

\bibitem[KeTa]{tao} M.\ Keel, T.\ Tao, \emph{Endpoint Strichartz estimates}, Amer.\ Math.\ J.\ 120 (1998), pp.\ 955--980.




\bibitem[KiYa]{kiya} H.\ Kitada, K. Yajima, \emph{A scattering theory for time-dependent long-range potentials}, Duke Math.\ J.\ Volume 49, Number 2 (1982), pp.\ 341--376.











\bibitem[NaSc1]{nakschl} K.\ Nakanishi, W.\ Schlag, \emph{Global dynamics above the ground state energy for the focusing nonlinear Klein-Gordon equation}, in preparation.

\bibitem[NaSc2]{nakschl2} K.\ Nakanishi, W.\ Schlag, \emph{Global dynamics above the ground state energy for the cubic NLS equation in 3D}, in preparation.





\bibitem[Rau]{rau} J.\ Rauch, \emph{Local decay of scattering solutions to Schršdinger's equation}, Comm.\ Math.\ Phys.\ Volume 61, Number 2 (1978), pp.\ 149-168.

\bibitem[ReSi1]{reesim1} B.\ Simon, M.\ Reed, \emph{Methods of Modern Mathematical Physics, I: Functional Analysis}, Academic Press, 1979.

\bibitem[ReSi3]{reesim} M.\ Reed, B.\ Simon, \emph{Methods of Modern Mathematical Physics, III: Scattering Theory}, Academic Press, 1979.

\bibitem[ReSi4]{reesim4} M.\ Reed, B.\ Simon, \emph{Methods of Modern Mathematical Physics, IV: Analysis of Operators}, Academic Press, 1979.

\bibitem[RoSc]{rodsch} I.\ Rodnianski, W.\ Schlag, \emph{Time decay for solutions of Schr\"odinger equations with rough and time-dependent potentials}, Invent.\ Math.\ 155 (2004), no.\ 3, pp.\ 451--513.




\bibitem[Sch]{schlag} W.\ Schlag, \emph{Stable Manifolds for an orbitally unstable NLS}, Annals of Mathematics, Vol.\ 169 (2009), No.\ 1, pp.\ 139--227.

\bibitem[Sch2]{schlag2} W.\ Schlag, \emph{Spectral theory and nonlinear partial differential equations: a survey}, Discrete Contin.\ Dyn.\ Syst.\ 15 (2006), No.\ 3, pp.\ 703--723.




\bibitem[Ste]{stein} E.\ Stein, \emph{Harmonic Analysis}, Princeton University Press, Princeton, 1994.





\bibitem[Tay]{taylor} M.\ E.\ Taylor, \emph{Tools for PDE.\ Pseudodifferential operators, paradifferential operators, and layer potentials}, Mathematical Surveys and Monographs, 81, American Mathematical Society, Providence, RI, 2000.




\bibitem[Yaj1]{yaj1} K.\ Yajima, \emph{Dispersive estimates for Schr\"{o}dinger equations with threshold resonance and eigenvalue}, Comm.\ Math.\ Physics, Vol.\ 259 (2005), No.\ 2, pp.\ 479--509.

\bibitem[Yaj2]{yaj2} K.\ Yajima, \emph{The $W^{k, p}$-continuity of Wave Operators for Schr\"{o}dinger Operators}, Proc.\ Japan Acad., 69, Ser.\ A (1993), pp.\ 94--99.



\bibitem[Yaj3]{yaj4} K.\ Yajima, \emph{The $L^p$ boundedness of wave operators for Schr\"{o}dinger operators with threshold singularities.\ I.\ The odd dimensional case}, J.\ Math.\ Sci.\ Univ.\ Tokyo 13 (2006), no.\ 1, pp.\ 43--93.




\bibitem[Wan]{wang} W.-M.\ Wang, \emph{Logarithmic Bounds on Sobolev Norms for Time Dependent Linear Schr\"{o}dinger Equations}, Communications in Partial Differential Equations, 33, pp.\ 2164--2179, 2008.

\end{thebibliography}
\end{document}